\documentclass[reqno]{amsart}

\usepackage{graphicx}

\usepackage{amsmath,amssymb}
\usepackage[hypertexnames=false]{hyperref}
\usepackage{cleveref}
\usepackage{autonum}
\usepackage[margin=3cm]{geometry}
\usepackage[svgnames,dvipsnames]{xcolor}
\usepackage{enumitem}
\usepackage{amsthm}
\usepackage{mathtools}

\colorlet{refkey}{pink!90!red}
\colorlet{labelkey}{JungleGreen!80!yellow}

\newtheorem{thm}{Theorem}[section]
\crefname{thm}{Theorem}{Theorems}
\newtheorem{lem}{Lemma}[section]
\crefname{lem}{Lemma}{Lemmas}
\newtheorem{cor}{Corollary}[section]
\crefname{cor}{Corollary}{Corollaries}

\theoremstyle{definition}

\crefname{assump}{Assumption}{Assumptions}

\crefname{rem}{Remark}{Remarks}

\crefname{defi}{Definition}{Definitions}

\numberwithin{equation}{section}

\newcommand{\bR}{\mathbb{R}}
\newcommand{\bN}{\mathbb{N}}

\newcommand{\cT}{\mathcal{T}}
\DeclareMathOperator{\diam}{diam}
\DeclareMathOperator{\dist}{dist}
\DeclareMathOperator{\supp}{supp}
\newcommand{\bdelta}{\bar{\delta}}
\newcommand{\OmghOmg}{{\Omega_h \setminus \Omega}}
\newcommand{\OmgOmgh}{{\Omega \setminus \Omega_h}}
\newcommand{\Omegah}{{\Omega_h}}
\newcommand{\dOmegah}{{\partial\Omega_h}}

\newcommand{\nh}{{n_h}}
\newcommand{\tGamma}{\tilde{\Gamma}}
\newcommand{\tu}{\tilde{u}}

\newcommand{\tw}{\tilde{w}}
\newcommand{\tf}{\tilde{f}}
\newcommand{\tg}{\tilde{g}}
\newcommand{\tz}{\tilde{z}}
\newcommand{\Omegahj}{{\Omega_{h,j}}}
\newcommand{\Qhj}{{Q_{h,j}}}
\newcommand{\trpl}{|\kern-0.25ex|\kern-0.25ex|}
\newcommand{\trplnorm}[2]{\trpl #1 \trpl_{#2}}

\newcommand{\cK}{\mathcal{K}}

\allowdisplaybreaks

\title[$L^\infty$-error estimates for parabolic FEM]{Maximum norm error estimates for the finite element approximation
of parabolic problems on smooth domains}
\author[Takahito Kashiwabara]{Takahito Kashiwabara}
\address{Graduate School of Mathematical Sciences, The University of Tokyo, 3-8-1 Komaba, Meguro, 153-8914 Tokyo, Japan}
\email{tkashiwa@ms.u-tokyo.ac.jp}
\author[Tomoya Kemmochi]{Tomoya Kemmochi}
\address{Department of Applied Physics,
Graduate School of Engineering, Nagoya University,
Furo-cho, Chikusa-ku, Nagoya, 464-8603, Aichi, Japan}
\email{kemmochi@na.nuap.nagoya-u.ac.jp}
\urladdr{https://t-kemmochi.github.io/}
\date{}
\thanks{The first author was supported by JSPS Grant-in-Aid for Young Scientists B (No.\ 17K14230).}

\begin{document}
\maketitle

\begin{abstract}
In this paper, we consider the finite element approximation for a parabolic problem on a smooth domain $\Omega \subset \bR^N$ with the inhomogeneous Neumann boundary condition.
We emphasize that the domain can be non-convex in general.
We implement the finite element method for this problem by constructing a family of polygonal or polyhedral domains $\{ \Omegah \}_h$ that approximate the original domain $\Omega$.
The main result of this study is the $L^\infty$-error estimate for this approximation.
We shall show that the convergence rate is not optimal for higher order elements since the symmetric difference $\Omega \bigtriangleup \Omegah$ is not empty in general.
In order to address the effect of the symmetric difference of domains, we introduce the tubular neighborhood of the original boundary $\partial\Omega$.
We will also present a slightly new approach to establish the $L^\infty$-error estimate.
Moreover, we present the smoothing property for the discrete parabolic semigroup and the spatially discretized maximal regularity as corollaries of the main result.
\end{abstract}

\section{Introduction}


In this paper, we consider the finite element method (FEM) for a parabolic problem on a bounded domain $\Omega \subset \bR^N$ with general $N \in \bN$, which can be non-convex.
We assume that the boundary $\partial\Omega$ is sufficiently smooth.
The target problem of the present paper is the parabolic equation on $\Omega$:
\begin{equation}
\begin{cases}
\partial_t u + Au = f, & \text{in } \Omega \times (0,T) =: Q_T, \\
\partial_n u = g, & \text{on } \partial\Omega \times (0,T) =: \Sigma_T, \\
u(0) = u_0, & \text{in } \Omega,
\end{cases}
\label{eq:parabolic}
\end{equation}
where $A = -\Delta + 1$, $f \colon \Omega \times (0,T) \to \bR$,
$g \colon \partial\Omega \times (0,T) \to \bR$, $u_0 \colon \Omega \to \bR$,
and $\partial_n$ denotes the outward normal derivative on $\partial\Omega$.
Although we can consider general (strongly) elliptic operators with smooth coefficients, we here address the operator $-\Delta + 1$ for simplicity.
We assume that given data $f$, $g$, and $u_0$ are sufficiently smooth.

The main purpose of the present study is the $L^\infty$-error estimate for the finite element semi-discretization of \eqref{eq:parabolic}.
In order to implement FEM on a smooth domain $\Omega$, we first approximate $\Omega$ by a polygonal domain.
Let $\Omega_h \subset \bR^N$ be a polygonal (or polyhedral) domain whose vertices lie on $\partial\Omega$.
We construct a conforming, shape-regular, and quasi-uniform triangulation $\cT_h$ of $\Omega_h$,
which is a family of open triangles (simplices in general) in $\Omegah$,
and we set $h_K = \diam K$ and $h = \max_{K \in \cT_h} h_K$.
We emphasize that $\Omega \bigtriangleup \Omega_h \ne \emptyset$ in general, where $\Omega \bigtriangleup \Omega_h$ is the symmetric difference or the \textit{boundary-skin layer}.
Then, we define $V_h \subset H^1(\Omega_h)$ as the conforming $P^k$-finite element space associated with $\cT_h$ for $k \ge 1$.
Now, the finite element approximation for \eqref{eq:parabolic} can be formulated as follows.
Find $u_h \in C^0([0,T]; V_h)$ that satisfies
\begin{equation}
\begin{cases}
(u_{h,t}(t), v_h)_{\Omega_h} + a_\Omegah(u_h(t), v_h) = (\tilde{f}(t), v_h)_{\Omega_h} + (\tilde{g}(t), v_h)_{\partial\Omega_h} , & \forall v_h \in V_h, \\
u_h(0) = u_{h,0},
\end{cases}
\label{eq:disc-parabolic}
\end{equation}
for each $t \in (0,T)$, where $u_{h,0} \in V_h$ is a given initial function and the bracket $(\cdot, \cdot)_D$ denotes the usual $L^2$-inner product over $D \subset \bR^N$ and $a_D(u,v) := (\nabla u, \nabla v)_D + (u,v)_D$.
Here, and hereafter, $\tf$ denotes an appropriate extension of $f$ in the sense of the Sobolev spaces.
Although the extension map can be different up to the regularity of the function, we will use the same notation.
This procedure is adopted in basic softwares for FEM such as FreeFEM++~\cite{Hec12} and FEniCS~\cite{LogMW12}, and thus it is important to investigate stability and error estimates for the approximation scheme \eqref{eq:disc-parabolic}.

One of our main results is the error estimate
\begin{multline}
\| \tu - u_h \|_{L^\infty(Q_{h,T})}
\le C h^{k+1} |\log h|^{\underline{k}} \| u \|_{L^\infty(0,T;W^{k+1,\infty}(\Omega))}  \\
+ C h^2 |\log h| \left(
\| u_t \|_{L^\infty(Q_T)} + \| u \|_{L^\infty(0,T;W^{2,\infty}(\Omega))}
\right),
\label{eq:intro-main}
\end{multline}
provided that $u \in W^{1,\infty}(0,T;L^\infty(\Omega)) \cap L^\infty(0,T;W^{k+1,\infty}(\Omega))$,
where $\underline{k} = 1$ if $k=1$ and $\underline{k} = 0$ otherwise.
The error estimate shall be given in a more general form (\cref{thm:main,cor:rate}).
The second line of \eqref{eq:intro-main} reflects the effect of the boundary-skin layer $\Omega \bigtriangleup \Omegah$.
Indeed, it does not appear if $\Omegah = \Omega$ (see e.g.,~\cite{SchTW98}).
The estimate \eqref{eq:intro-main} implies that the convergence rate of the scheme \eqref{eq:disc-parabolic} is $O(h^2|\log h|)$ even for higher order elements, since we are approximating the boundary by ``piecewise linear'' shapes.

In addition to the error estimate \eqref{eq:intro-main}, we shall show the smoothing property and maximal regularity for the discrete Laplace operator $A_h$,
as discussed in \cite{SchTW98,ThoW00,Gei06,Li15}  (see \cref{thm:semigroup,thm:dmr} and Section~\ref{sec:cor}).
Here, we define $A_h$ by
\begin{equation}
(A_h u_h, v_h)_{\Omega_h} = a_\Omegah(u_h,v_h),
\quad \forall u_h,v_h \in V_h,
\label{eq:disc-lap}
\end{equation}
which is a discrete analog of Green's formula.
We shall show that the estimate
\begin{equation}
\| u_h(t) \|_{L^q(\Omega_h)} + t \| \partial_t u_h(t) \|_{L^q(\Omega_h)} \le C e^{-ct} \| u_{h,0} \|_{L^q(\Omega_h)},
\quad \forall t > 0,
\end{equation}
holds for $q \in [1,\infty]$ when $f \equiv 0$ and $g \equiv 0$ (\cref{thm:semigroup}), and
\begin{equation}
\| A_h u_h \|_{L^p(0,T; L^q(\Omega_h))} + \| \partial_t u_h \|_{L^p(0,T; L^q(\Omega_h))} \le C \| f_h \|_{L^p(0,T; L^q(\Omega_h))},
\end{equation}
holds for $p,q \in (1,\infty)$ when $u_{h,0} \equiv 0$, $g \equiv 0$, and $\tf = f_h \in L^p(0,T;V_h)$ (\cref{thm:dmr}).
In these estimates, the effect of the boundary-skins can be considered as just perturbation,
and thus we can obtain the same estimates as in \cite[Theorem~2.1]{SchTW98} and \cite[Theorem~3.2]{Gei06}.

In the context of FEM, the domain $\Omega$ is usually assumed to be a polygonal or polyhedral domain so that triangulations can be exactly implemented.
However, it is known that the regularity of the solution cannot be guaranteed if there exist corners in the boundary of the domain (see e.g.~\cite{Gri85}).
Lack of regularity of solutions is troublesome in numerical analysis for partial differential equations, especially for nonlinear problems.
For example, in \cite{Sai12,ZhoS17}, finite element and finite volume schemes for the Keller-Segel system on polygonal domains are considered.
In their error estimates (\cite[Theorem~2.4]{Sai12} and \cite[Theorem~3.1]{ZhoS17}), the convergence rate in $L^\infty(0,T; L^p(\Omega))$-norm is $O(h^{1-N/p})$, in contrast to the expected rate $O(h)$, where $h$ is the mesh size.
This shortcoming is caused by the corner singularity of the boundary.
Indeed, it is shown that the convergence rate is $O(h)$ if the boundary is smooth \cite[Section~5.1]{Sai12}.

In view of the theory of nonlinear partial differential equations, appropriate regularity, such as smoothing property and maximal regularity, is essential for analysis of equations.
Therefore, it is natural to assume the boundary is smooth, and consequently,
it is important to consider FEM for such problems.
Moreover, keeping application to nonlinear evolution equations in mind, it is valuable to derive error estimates in various norms such as $L^\infty(Q_T)$ and $L^p(0,T;L^q(\Omega))$.
Indeed, there are many results on FEM for parabolic problems that have succeeded in deriving error estimates in the framework of analytic semigroups (e.g., \cite{EriJL98,Sai12,ZhoS17}) and maximal regularity (e.g., \cite{Gei07,LiS15,LeyV17,KemS17}).

In the literature of FEM, 
there are several strategy
to overcome the loss of accuracy induced by the corner singularity of the boundary.
The classical one is using the isoparametric FEM \cite{Cia78}.
However, this method requires delicate analysis, especially for the higher order and higher dimensional cases.
Recently, the isogeometric analysis (IGA) \cite{CorHB09,BazTT13} is widely used to solve partial differential equations on smooth domains, which is based on the NURBS basis \cite{PieT97,Far99}.
This method can represent the boundary exactly for a class of domains and thus there is no need to consider errors on approximation of the boundary.
It has, nevertheless, a problem on numerical quadrature since this method is based on coordinate transformations by rational functions.
Therefore, we should take a great care of errors on numerical quadrature.
An alternative approach is to modify the bilinear form with a usual triangulation mentioned above or a so-called background mesh (e.g.,~\cite{CocS12,CocQS14,CocS14,MaiS17,BocCPG17}).
These methods are implementable and give optimal order estimates.
However, the implementation requires more information on the geometry of the boundary such as normal vectors.
In contrast to these studies, we address the simplest scheme \eqref{eq:disc-parabolic}.

There are many studies on the $L^\infty$-analysis for FEM for parabolic problems (e.g., \cite{BraSTW77,SchTW98,ThoW00,Li15,LeyV16} and references therein).
In particular, \cite{SchTW98} gives a general method for $L^\infty$-analysis of FEM for parabolic problems via the regularized Green's function.
All of them assume that the boundary condition is homogeneous and the domain is smooth and convex.
For the Dirichlet condition (e.g., \cite{BraSTW77,ThoW00}), they consider a family of polynomial (or polyhedral) domains $\{ \Omega_h \}_h$ whose vertices lie in $\partial\Omega$, and introduced a space of piecewise polynomials associated with a triangulation of $\Omega_h$ that vanishes on $\partial\Omega_h$.
Then, they extend each functions in such a space by zero in $\Omega \setminus \Omega_h$.
Therefore, piecewise polynomials can be viewed as functions in $H^1_0(\Omega)$, yet this procedure is available for convex domains and for homogeneous Dirichlet problems.
For the Neumann problems (e.g., \cite{SchTW98,Li15}), they assumed that the domain is exactly triangulated.
That is, they extended piecewise polynomial functions by considering pie-shaped element near the boundary.
However, this extension is unavailable for the three-dimensional case, even if the domain is convex as pointed-out in \cite[page~1356]{SchTW98}.
The same assumptions are imposed in the literature on discrete maximal regularity on smooth domains \cite{Gei06,Gei07,Li15}.


In contrast to these studies, we never assume that $\Omega$ is convex and thus $\Omega \bigtriangleup \Omegah \ne \emptyset$.
Therefore, we should take care of the effect of boundary-skins, as mentioned above.
In order to address the integration over $\Omega \bigtriangleup \Omegah$, we introduce the tubular neighborhood of $\partial\Omega$.
As in the analysis of FEM for elliptic equations, the Galerkin orthogonality (or compatibility) is essential in $L^\infty$-analysis for FEM of parabolic problems (cf.~\cite{SchTW98}).
However, since $\Omega \bigtriangleup \Omegah \ne \emptyset$, it does not hold in general and there appear additional terms (see \cref{lem:ago}).
We shall address these terms using the tubular neighborhood as in the elliptic case discussed in our previous paper \cite{KasK17}.
This procedure is available for inhomogeneous Neumann boundary conditions, in contrast to previous work addressing the homogeneous case only.

The main strategy of the proof of \eqref{eq:intro-main} is similar to \cite{SchTW98}.
That is, we introduce the regularized delta function, regularized Green's function $\Gamma$, and its finite element approximation $\Gamma_h$.
Then, we reduce the $L^\infty$-error estimate to the $L^1$-type estimates for $F = \Gamma_h - \tGamma$ (\cref{lem:green}).
We will introduce a parabolic dyadic decomposition $Q_{h,j}$ (see \eqref{eq:dyadic1}) and we address the norms of $F$ over each $Q_{h,j}$.
However, in the proof of the estimates for $F$, we shall take a slightly different approach.
In \cite{SchTW98}, they also introduce the parabolic dyadic decomposition and consider a local energy error estimate with a kick-back argument.
For this purpose, they show strong super-approximation property for the discrete space $V_h$ \cite[Section~5]{SchTW98}.
The argument of \cite{ThoW00} is similar and they consider a delicate estimate with a special cut-off function \cite[pages~387--388]{ThoW00}.
Finally, local estimates are merged with respect to the dyadic decomposition, and the $L^1$-estimates are obtained.
In contrast to these arguments, we will use the kick-back argument after summation.
Our strategy does not require the strong super-approximation property and special cut-off functions.
Therefore, the present study provides an alternative proof for $L^\infty$-analysis of FEM for parabolic problems.


The rest of this paper is organized as follows.
In Section~\ref{sec:main}, we present our notation and state the main results.
In Section~\ref{sec:outline}, we give the outline of the proof of the main theorem.
The lemmas stated in this section are proved in subsequent sections.
In Section~\ref{sec:preliminaries}, we summarize preliminary results on FEM, tubular neighborhood, and the regularized Green's functions.
The estimates stated in subsection~\ref{subsec:tubular} will be used repeatedly in this paper.
Section~\ref{sec:proof} is devoted to the proof of the $L^\infty$-error estimate.
However, we will postpone the proof of $L^1$-estimates for $F$, which is given in Section~\ref{sec:local}.
As explained above, we shall propose a slightly new approach for the $L^1$-estimates.
In Section~\ref{sec:dual}, we show the local $L^2$-estimates for $F$ by the duality argument.
Finally, we will present the proofs of the smoothing property and the maximal regularity for the discrete elliptic operator $A_h$ in Section~\ref{sec:cor}.
Throughout this paper, the symbol $C$ denote generic constants, which may be different in each appearance.

\section{Notation and main results}
\label{sec:main}

Let $\Omega \subset \bR^N$ be a bounded domain with general $N \in \bN$.
Without loss of generality, we may assume $\diam \Omega \le 1$.
We also suppose that $\partial \Omega$ is sufficiently smooth.
The target problem of the present paper is the parabolic equation \eqref{eq:parabolic} on $\Omega$
with smooth data $f \colon \Omega \times (0,T) \to \bR$,
$g \colon \partial\Omega \times (0,T) \to \bR$, and $u_0 \colon \Omega \to \bR$.
The weak form of the problem \eqref{eq:parabolic} is described as follows.
Find $u \in C^0((0,T);V_h)$ that satisfies
\begin{equation}
\begin{cases}
(u_t(t), v)_\Omega + a_\Omega(u(t), v) = (f(t),v)_\Omega + (g(t),v)_{\partial\Omega}, & \forall v \in H^1(\Omega),\ t \in (0,T), \\
u(0) = u_0,
\end{cases}
\label{eq:weak-parabolic}
\end{equation}
where $(\cdot, \cdot)_D$ denotes the $L^2$-inner product over the domain $D \subset \bR^N$ and
\begin{equation}
a_D(u,v) := (\nabla u, \nabla v)_D + (u,v)_D.
\end{equation}

Let us next consider the finite element approximation of \eqref{eq:parabolic}.
To do that, we first approximate the domain $\Omega$ by polygonal (or polyhedral) domains.
Let $\Omega_h \subset \bR^N$ be a polygonal domain and $\cT_h$ be a triangulation, i.e., family of (open) triangles (simplexes in general), of $\Omega_h$ with $h = \max_{K \in \cT_h} \diam K$.
Throughout this paper, we assume that $\Omega_h$ and $\cT_h$ enjoy the following conditions.
\begin{itemize}
\item All of the vertices of $\partial\Omega_h$ belong to $\partial\Omega$.
\item There is no triangle whose vertex belongs to $\partial\Omega_h \setminus \partial\Omega$.
\item For each simplex $K \in \cT_h$, $K \cap \Omega \ne \emptyset$.
\end{itemize}
Moreover, we suppose that $\cT_h$ is shape-regular and quasi-uniform.
Note that $\Omega_h \bigtriangleup \Omega \ne \emptyset$ in general and the identity
\begin{equation}
\int_\Omega f dx - \int_{\Omega_h} f dx = \int_{\OmgOmgh} f dx - \int_{\OmghOmg} f dx
\label{eq:gap}
\end{equation}
holds.
We also set $Q_{h,T} := \Omega_h \times (0,T)$ 
and $\Sigma_{h,T} := \dOmegah \times (0,T)$. 

Let $V_h \subset H^1(\Omega_h)$ be the conforming $P^k$-finite element space associated with $\cT_h$ ($k \ge 1$).
If $f$, $g$, and $u_0$ are sufficiently smooth, then we can extend these functions over $\Omega_h$ in the sense of Sobolev spaces.
We denote one of such extensions by $\tilde{f}$ and so on.
Then, we can formulate the finite element approximation of \eqref{eq:parabolic} as follows.
Find $u_h \in C^0([0,T]; V_h)$ that satisfies \eqref{eq:disc-parabolic}
for each $t \in (0,T)$ and a given initial function $u_{h,0} \in V_h$.
The main theorem of the present paper is the following $L^\infty$-error estimate for the problem \eqref{eq:disc-parabolic}.
We emphasize that the extension $\tilde{u}$ is arbitrary.

\begin{thm}[Maximum norm error estimate]
\label{thm:main}
Let $\cT_h$ be a shape-regular and quasi-uniform triangulation of $\Omega$.
Let $u$ and $u_h$ be solutions of \eqref{eq:parabolic} and \eqref{eq:disc-parabolic}, respectively, for given data $f$, $g$, $u_0$, and $u_{h,0}$.
Assume $u \in W^{1,\infty}(0,T; L^\infty(\Omega)) \cap L^\infty(0,T; W^{2,\infty}(\Omega))$.
Then, we have
\begin{multline}
\| \tilde{u} - u_h \|_{L^\infty(Q_{h,T})}
\le C \bigg[ \| \tilde{u}_0 - u_{h,0} \|_{L^\infty(\Omega_h)}
+ |\log h|^{\underline{k}} \inf_{\chi \in C^0([0,T]; V_h)} \| \tilde{u} - \chi \|_{L^\infty(Q_{h,T})} \\
+ h^2 |\log h|
\left( \|u\|_{L^\infty(0,T;W^{2,\infty}(\Omega))}
     + \|u_t\|_{L^\infty(0,T; L^\infty(\Omega))} \right)
\bigg],
\label{eq:best}
\end{multline}
where $\underline{k} = 1$ if $k=1$ and $\underline{k} = 0$ otherwise.
Here, the constant $C$ is independent of $h$, $u$, $u_h$, $f$, $g$, $u_0$, $u_{h,0}$ and $T$.
\end{thm}

Since $V_h$ is the space of piecewise polynomials of degree $k$, we can determine the convergence rate from the above estimate.
The rate is not optimal even for higher order elements due to the boundary-skin, in contrast to the convex case \cite{SchTW98}.

\begin{cor}[Convergence rate]
\label{cor:rate}
In addition to the hypotheses in \cref{thm:main}, we assume that $u \in C^0([0,T]; W^{l,\infty}(\Omega))$ for some $2 \le l \le k+1$.
Then, we have
\begin{multline}
\| \tilde{u} - u_h \|_{L^\infty(Q_{h,T})}
\le C \bigg[  h^{l} |\log h|^{\underline{k}} \| u \|_{L^\infty(0,T;W^{l,\infty}(\Omega))} \\
+ h^2 |\log h|
\left( \|u\|_{L^\infty(0,T;W^{2,\infty}(\Omega))}
     + \|u_t\|_{L^\infty(0,T; L^\infty(\Omega))} \right)
\bigg],
\end{multline}
where $C>0$ is independent of $h$, $u$, $u_h$, $f$, $g$, $u_0$, $u_{h,0}$, and $T$.
\end{cor}

According to \cite{SchTW98} and \cite{Gei06}, we can obtain the stability, analyticity, and the (spatially) discrete maximal regularity results for the discrete heat semigroup as follows.
Recall that $A_h$ is the discrete Laplace operator defined by \eqref{eq:disc-lap}.

\begin{thm}[Stability and analyticity of the discrete semigroup]
\label{thm:semigroup}
Let $q \in [1,\infty]$ and $\cT_h$ be a shape-regular and quasi-uniform triangulation of $\Omega$.
Let $u_h$ be the solution of \eqref{eq:disc-parabolic} for $f=0$ and $g=0$.
Then, we have
\begin{equation}
\| u_h(t) \|_{L^q(\Omega_h)} + t \| \partial_t u_h(t) \|_{L^q(\Omega_h)} \le C e^{-ct} \| u_{h,0} \|_{L^q(\Omega_h)},
\quad \forall t > 0,
\label{eq:semigroup}
\end{equation}
where $C>0$ and $c>0$ are independent of $h$, $u_h$, $u_{h,0}$, and $t$.
\end{thm}

\begin{thm}[Discrete maximal regularity]
\label{thm:dmr}
Let $p, q \in (1,\infty)$ and $\cT_h$ be a shape-regular and quasi-uniform triangulation of $\Omega$.
Let $u_h$ be the solution of \eqref{eq:disc-parabolic} for $u_{h,0}=0$, $g=0$, and $\tilde{f} = f_h \in L^p(0,T; V_h)$.
Then, we have
\begin{equation}
\| A_h u_h \|_{L^p(0,T; L^q(\Omega_h))} + \| \partial_t u_h \|_{L^p(0,T; L^q(\Omega_h))} \le C \| f_h \|_{L^p(0,T; L^q(\Omega_h))},
\label{eq:dmr}
\end{equation}
where $C>0$ is independent of $h$, $u_h$, $f_h$, and $T$.
\end{thm}
These two theorems are shown in Section~\ref{sec:cor}.

\section{Outline of the proof}
\label{sec:outline}

In this section, we present the outline of the proof of \cref{thm:main}.
Precise arguments are given in subsequent sections.

As in the previous work on maximum-norm estimates for FEM, we first introduce the regularized delta and Green's functions.
Fix $K_0 \in \cT_h$ and $x_0 \in K_0 \subset \Omega_h$ arbitrarily.
Then, we can construct a smooth function $\bar{\delta} = \bar{\delta}_{x_0} \in C^\infty_0(K_0)$ that fulfills
\begin{equation}
P(x_0) = (P, \bar{\delta})_{K_0}, \quad \forall P \in \mathcal{P}^k(K_0),
\label{eq:delta}
\end{equation}
where $\mathcal{P}^k(K_0)$ is the set of all polynomials of degree at most $k$ over $K_0$.
For construction, see \cite[Appendix]{SchSW96}.
We then define the regularized Green's function $\Gamma$ as the solution of the homogeneous problem
\begin{equation}
\begin{cases}
\partial_t \Gamma + A \Gamma = 0, & \text{in } Q_T, \\
\partial_n \Gamma = 0, & \text{on } \partial\Omega \times (0,T), \\
\Gamma(0) = \bdelta, & \text{in } \Omega.
\end{cases}
\label{eq:gamma}
\end{equation}
Note that $\Gamma \in C^\infty(\overline{Q_T})$ since $\bdelta$ and $\partial\Omega$ are sufficiently smooth.
Furthermore, we define $\Gamma_h$ as the finite element approximation of $\Gamma$ as follows.
\begin{equation}
\begin{cases}
(v_h, \Gamma_{h,t}(t))_{\Omega_h} + a_\Omegah(v_h, \Gamma_h(t)) = 0 , & \forall v_h \in V_h, \ t \in (0,T), \\
\Gamma_h(0) = P_h \bdelta.
\end{cases}
\label{eq:disc-gamma}
\end{equation}
We finally set $F := \Gamma_h - \tilde{\Gamma}$, which is a function defined on $\Omegah$.

Now, let $u \in W^{1,\infty}(0,T; L^\infty(\Omega)) \cap L^\infty(0,T; W^{2,\infty}(\Omega))$ be the solution of \eqref{eq:parabolic} and $u_h \in C^0([0,T]; V_h)$ be that of \eqref{eq:disc-parabolic}.
From the stability result of \cref{thm:semigroup}, we may assume $u_{h,0} = P_h \tu_0$, where $P_h$ is the orthogonal projection in $L^2(\Omegah)$.
Moreover, we may assume $T \le 1$ (see Subsection~\ref{subsec:long-time} for the case $T \ge 1$).
Then, owing to \eqref{eq:delta}, we have
\begin{equation}
(\tilde{u}-u_h)(x_0,T) = (\tilde{u}-P_h\tilde{u})(x_0,T) 
+ ((P_h\tilde{u}-u_h)(T), \bdelta)_{\Omega_h},
\end{equation}
which implies
\begin{equation}
|(\tilde{u}-u_h)(x_0,T)|
\le C \inf_{\chi \in C^0([0,T]; V_h)} \|\tilde{u}-\chi\|_{L^\infty(Q_{h,T})} 
+ |((P_h\tilde{u}-u_h)(T), \bdelta)_{\Omega_h}|
\label{eq:pointwise-0}
\end{equation}
for arbitrary $\chi \in C^0([0,T]; V_h)$, since $P_h$ is uniformly bounded in $L^\infty(\Omega_h)$.

We will address the last term of \eqref{eq:pointwise-0} and 
show the following estimate (cf.~\cref{lem:reduce}, \eqref{eq:E0}, and \cref{lem:Ej}).
We remark that the third line of \eqref{eq:pointwise} is induced by the boundary-skin layer of the domain.
\begin{lem} \label{lem:pointwise}
If $u_{h,0} = P_h \tu_0$, then we have
\begin{multline}
|((P_h\tilde{u}-u_h)(T), \bdelta)_{\Omega_h}| \\
\le C \left( |\log h|^{\underline{k}} + \| F_t \|_{L^1(Q_{h,T})} + h^{-1} |\log h|^{-\underline{k}} \| F \|_{L^1(0,T;W^{1,1}(\Omegah))} \right)
\| \tilde{u} - \chi \|_{L^\infty(Q_{h,T})} \\
+ C h^2 \left( |\log h| + h^{-1} \| F \|_{L^1(0,T;W^{1,1}(\Omegah))} \right)  \left( \| u \|_{L^\infty(0,T;W^{2,\infty}(\Omega))} + \| u_t \|_{L^\infty(Q_T)} \right), 
\label{eq:pointwise}
\end{multline}
for any $\chi \in C^0([0,T]; V_h)$.
\end{lem}

Therefore, it suffices to show the following estimate, which is addressed in Subsection~\ref{subsec:proof}.
\begin{lem}\label{lem:green}
Assume $T \le 1$.
Then, we have
\begin{equation}
\| F_t \|_{L^1(Q_{h,T})} + h^{-1} |\log h|^{-\underline{k}} \| F \|_{L^1(0,T; W^{1,1}(\Omega_h))}
\le C,
\label{eq:L1-F}
\end{equation}
where $\underline{k}$ is the same symbol as in \cref{thm:main}.
\end{lem}

Here, we present the outline of the proof of \cref{lem:green}.
In order to establish \eqref{eq:L1-F},
we introduce the parabolic dyadic decomposition
according to \cite{SchTW98}.
Let $d_j := 2^{-j-1}$ for $j \in \bN_0 = \{0\} \cup \bN$.
We fix $J_* \in \bN$ such that $C_* h \le d_{J_*} \le 2C_* h$ for some $C_* \ge 1$, which is determined later independently of $h$.
By definition, $J_* \approx |\log h|$.
We remark that $h \le C_*^{-1} d_{J_*} \le C_*^{-1} d_j$ and
\begin{equation}
\sum_{j=0}^{J_*} \left( \frac{h}{d_j} \right)^r \le C
\label{eq:sum}
\end{equation}
for $r>0$, where $C$ depends only on $r$.

For the fixed $x_0 \in \Omega_h$ as above,
let $\rho(x,t) := \max\{ |x-x_0| ,\, \sqrt{t} \}$ and
\begin{align}
\Omegahj &= \{ x \in \Omegah \mid d_j \le |x-x_0| \le 2 d_j \}, &
\Omega_{h,*} &= \{ x \in \Omegah \mid |x-x_0| \le d_{J_*} \},
\label{eq:dyadic1}
\\
\Qhj &= \{ (x,t) \in Q_{h,T} \mid d_j \le \rho(x,t) \le 2d_j \}, &
Q_{h,*} &= \{ (x,t) \in Q_{h,T} \mid \rho(x,t) \le d_{J_*} \}.
\label{eq:dyadic2}
\end{align}
Then, it is clear that
\begin{equation}
\Omegah = \left( \bigcup_{j=1}^{J_*} \Omegahj \right) \cup \Omega_{h,*}, \quad
Q_{h,T} = \left( \bigcup_{j=1}^{J_*} \Qhj \right) \cup Q_{h,*}.
\end{equation}
We also set $\Omega'_{h,j} = \Omega_{h,j-1} \cup \Omega_{h,j} \cup \Omega_{h,j+1}$
and $Q'_{h,j} = Q_{h,j-1} \cup \Qhj \cup Q_{h,j+1}$ for later use.
Note that the summation with respect to $Q'_{h,j}$ is controlled in terms of $Q_{h,j}$.
Indeed, one can see 
\begin{equation}
\sum_{j=0}^{J_*} d_j^r \| w \|_{L^2(Q'_{h,j})}
\le 
3 \cdot 2^r \left( d_{J_*}^r \| w \|_{L^2(Q_{h,*})} + 
\sum_{j=0}^{J_*} d_j^r \| w \|_{L^2(Q_{h,j})}
\right) 
\label{eq:summation-Qhj-prime}
\end{equation}
for any $r>0$ by the definition of $Q'_{h,j}$ and $d_j$.

Moreover, in order to address the effect of the boundary-skin, 
we define the tubular neighborhood of the boundary $T(\varepsilon)$ by
\begin{equation}
T(\varepsilon) := \{ x \in \bR^N \mid \dist(x, \partial\Omega) < \varepsilon \}
\label{eq:tubular}
\end{equation}
for $\varepsilon > 0$, 
where $\dist(x,D) = \inf_{y \in D} |x-y|$ for $x \in \bR^N$ and $D \subset \bR^N$.
In fact, we can set $\varepsilon = O(h^2)$ since $\cT_h$ is quasi-uniform (see Subsection~\ref{subsec:tubular}).
Further, we set $L_T(\varepsilon) := T(\varepsilon) \times (0,T)$.

We here introduce space-time norms of $L^2$-type.
For $Q \subset \bR^{N+1}$ and $l \in \bN$, we define
\begin{equation}
\trpl v \trpl_Q := \| v \|_{L^2(Q)},
\quad
\trpl v \trpl_{l,Q} := \sum_{|\alpha| \le l} \| \nabla^\alpha v \|_{L^2(Q)},
\end{equation}
and we also write
\begin{equation}
\| v \|_D = \| v \|_{L^2(D)},
\quad
\| v \|_{l,D} = \| v \|_{H^l(D)}
\end{equation}
for $D \subset \bR^N$.
Then, the $L^1$-norms of $F$ can be bounded by weighted $L^2$-norms by the H\"older inequality and  we have
\begin{equation}
\| F \|_{L^1(0,T; W^{1,1}(\Omegah))} 
\le Ch + C \sum_j d_j^{\frac{N}{2}+1} \trplnorm{F}{1,Q_{h,j}}
\label{eq:outline-3}
\end{equation}
and
\begin{equation}
\| F_t \|_{L^1(Q_{h,T})} 
\le C + C \sum_j d_j^{\frac{N}{2}+1} \trplnorm{F_t}{Q_{h,j}},
\end{equation}
owing to the innermost estimates (cf.~\cref{lem:energy-F}) and $|Q_{h,j}| \approx d_j^{N+2}$.

Local terms
$\trplnorm{F}{1,Q_{h,j}}$ and $\trplnorm{F_t}{Q_{h,j}}$
will be addressed by the following two lemmas.
We again emphasize that the term $G_j$ (and the term involving $F(T)$) in \eqref{eq:local-F} and 
the third line of \eqref{eq:dual} indicate the effect of the boundary-skin layer of the domain.

\begin{lem}\label{lem:local-F}
For arbitrarily small positive number $\theta$, we have
\begin{multline}
\theta \trplnorm{F_t}{Q_{h,j}} + d_j^{-1} \trplnorm{F}{1,Q_{h,j}} + \theta \lambda_j \| F(T) \|_{1,D_{h,j}} \\
\le
  C_0 \theta 
  \left( \theta \trplnorm{F_t}{Q'_{h,j}}
+ d_j^{-1} \trplnorm{F}{1,Q'_{h,j}}
+ \theta \lambda'_j \| F(T) \|_{1,D'_{h,j}} \right)\\
+ C \left( I_j + X_j + G_j \right)
+ C d_j^{-2} \trplnorm{F}{Q'_{h,j}} ,
\label{eq:local-F}
\end{multline}
for some constants $C_0>0$ and $C>0$ independent of $h$, $j$, and $T$, where 
\begin{align}
D_{h,j} &:= Q_{h,j} \cap (\Omegah \times \{ T \}), \quad
D'_{h,j} := Q'_{h,j} \cap (\Omegah \times \{ T \}), \\
\lambda_j &:= \begin{cases}
1, & D_{h,j} \ne \emptyset, \\
0, & \text{otherwise},
\end{cases} \quad
\lambda'_j := \begin{cases}
1, & D'_{h,j} \ne \emptyset, \\
0, & \text{otherwise},
\end{cases} \\
I_j &:= \| F(0) \|_{1, \Omega'_{h,j}} + d_j^{-1} \| F(0) \|_{\Omega'_{h,j}}, \\
X_j &:= d_j \trplnorm{\zeta_t}{1,Q'_{h,j}} + \trplnorm{\zeta_t}{Q'_{h,j}}
+ d_j^{-1} \trplnorm{\zeta}{1,Q'_{h,j}} + d_j^{-2} \trplnorm{\zeta}{Q'_{h,j}}, \\
G_j &:= 
\begin{multlined}[t]
hd_j^{\frac{3}{2}} \trplnorm{\tGamma_{tt} + A\tGamma_t}{L_T(\varepsilon) \cap Q'_{h,j}}
+ hd_j^{-\frac{3}{2}} \trplnorm{\tGamma_t + A\tGamma}{L_T(\varepsilon) \cap Q'_{h,j}} \\
+ d_j^{\frac{3}{2}} \trplnorm{\partial_{n_h} \tGamma_t}{\Sigma_{h,T} \cap Q'_{h,j}}
+ d_j^{-\frac{3}{2}} \trplnorm{\partial_{n_h} \tGamma}{\Sigma_{h,T} \cap Q'_{h,j}}
\end{multlined}
\end{align}
and $\zeta = \tGamma - I_h \tGamma$.
\end{lem}
%
\begin{lem}\label{lem:duality}
There exists $C>0$ independent of $C_*$, $h$, and $j$ that satisfies
\begin{multline}
\trplnorm{F}{Q_{h,j}}
\le C h^2 d_j^{-\frac{N}{2}-1} + C \sum_i \left( h^2 \trplnorm{F_t}{Q_{h,i}} + h \trplnorm{F}{1,Q_{h,i}} \right) \min \left\{ \left( \frac{d_i}{d_j} \right)^{\frac{N}{2}+1}, \left( \frac{d_j}{d_i} \right)^{\frac{N}{2}+1} \right\} \\
+ C hd_j^{-\frac{N}{2} + \frac{1}{2}} 
+ C h \left( d_j^{-1} \trplnorm{F}{Q'_{h,j}} + \trplnorm{F}{1,Q'_{h,j}} \right)
+ C h d_j^{-\frac{N}{2}} \| F \|_{L^1(0,T; W^{1,1}(\Omega_h))}.
\label{eq:dual}
\end{multline}
\end{lem}

Now, we complete the sketch of the proof of \cref{lem:green}.
Multiplying \eqref{eq:local-F} by $d_j^{\frac{N}{2}+2}$, summing up with respect to $j$, applying \eqref{eq:summation-Qhj-prime},
and making $\theta$ small enough, we can show that (see~\eqref{eq:green-5})
\begin{equation}
\sum_j \left( d_j^{\frac{N}{2}+2} \trplnorm{F_t}{Q_{h,j}} + d_j^{\frac{N}{2}+1} \trplnorm{F}{1,Q_{h,j}} \right) 
\le Ch|\log h|^{\underline{k}} 
+ C \sum_j d_j^{\frac{N}{2}} \trplnorm{F}{Q_{h,j}}.
\label{eq:outline-1}
\end{equation}
Moreover, \eqref{eq:dual} implies (see~\eqref{eq:dual-F_t})
\begin{multline}
\sum_j d_j^{\frac{N}{2}} \trplnorm{F}{Q_{h,j}} 
\le C h
+ C C_*^{-1} \sum_j \left( 
d_j^{\frac{N}{2}+2} \trplnorm{F_t}{Q_{h,j}} 
+ d_j^{\frac{N}{2}+1} \trplnorm{F}{1,Q_{h,j}} 
\right) \\
+ C h |\log h| \| F \|_{L^1(0,T; W^{1,1}(\Omega_h))}.
\label{eq:outline-2}
\end{multline}
Therefore, substituting \eqref{eq:outline-2} into \eqref{eq:outline-1} and
making $C_*$ large enough, we can obtain
\begin{equation}
\sum_j d_j^{\frac{N}{2}+1} \trplnorm{F}{1,Q_{h,j}}
\le  Ch|\log h|^{\underline{k}} 
+ C h |\log h| \| F \|_{L^1(0,T; W^{1,1}(\Omega_h))}.
\end{equation}
Finally, going back to \eqref{eq:outline-3} and letting $h$ small enough, 
we can establish
\begin{equation}
\| F \|_{L^1(0,T; W^{1,1}(\Omega_h))}
\le  Ch|\log h|^{\underline{k}} .
\end{equation}
Similarly, we can show
\begin{equation}
\| F_t \|_{L^1(Q_{h,T})} \le  C
\end{equation}
and thus we can complete the proof of \eqref{eq:L1-F}.
Returning to \cref{lem:pointwise}, we obtain the desired estimate \eqref{eq:best}.
The rest of the present paper is devoted to the proofs of the above estimates.

\section{Preliminaries}
\label{sec:preliminaries}

\subsection{Projection and interpolation}

We introduce projection and interpolation operators associated with $V_h$.
As mentioned above, we denote the $L^2(\Omegah)$-projection by $P_h$.
The node-wise interpolation operator is denoted by~$I_h$.
Furthermore, we construct a ``quasi-interpolation'' operator $\tilde{I}_h$ acting on the Sobolev space $W^{1,1}(\Omegah)$, whereas $I_h$ acts on the space of continuous functions.
For construction, see \cite[Section~5]{Gei06} (especially, definition of $\tilde{I}^N_h$).
For these operators, the following stability and error estimates hold.
The proofs can be found in \cite{BreS08} and \cite{Gei06} (see also \cite[Lemma~2.1]{ThoW00}).

\begin{lem}\label{lem:approx}
Assume that $\cT_h$ is shape-regular and quasi-uniform.
\begin{enumerate}[label=\upshape (\roman{*})]
\item For each $p \in [1,\infty]$, we have
\begin{alignat}{2}
\| P_h v \|_{L^p(\Omegah)}
&\le C \|v\|_{L^p(\Omegah)},
& \quad \forall v
& \in L^p(\Omegah), \\
\| P_h v \|_{W^{1,p}(\Omegah)}
&\le C \|v\|_{W^{1,p}(\Omegah)},
& \quad \forall v
&\in W^{1,p}(\Omegah), \\
\| v - P_h v \|_{L^p(\Omegah)}
&\le Ch^2 \|v\|_{W^{2,p}(\Omegah)},
& \quad \forall v
&\in W^{2,p}(\Omegah), \\
\| v - P_h v \|_{W^{1,p}(\Omegah)}
&\le Ch \|v\|_{W^{2,p}(\Omegah)},
& \quad \forall v
&\in W^{2,p}(\Omegah).
\end{alignat}
\item Let $0 \le l \le k$ be integers.
Then, for each $K \in \cT_h$, we have
\begin{equation}
\| \nabla^l (v - I_h v) \|_{L^\infty(K)} \le C h^{k-l} \|\nabla^k v\|_{L^\infty(K)},
\quad \forall v \in C^k(\overline{K}),
\end{equation}
where $C$ is independent of $h$, $K$, and $v$.
\item Let $K \in \cT_h$ and $M_K := \bigcup \{ \overline{T} \in \cT_h \mid \overline{T} \cap \overline{K} \ne \emptyset \}$.
Then, for each $p \in [1,\infty]$, we have
\begin{alignat}{2}
\| v-\tilde{I}_h v\|_{L^p(K)}
&\le C h \| \nabla v \|_{L^p(M_K)},
& \quad \forall v
& \in W^{1,p}(M_K), \\
\| v-\tilde{I}_h v\|_{L^p(K)}
&\le C h^2 \| \nabla^2 v \|_{L^p(M_K)},
& \quad \forall v
& \in W^{2,p}(M_K), \\
\| \nabla (v-\tilde{I}_h v)\|_{L^p(K)}
&\le C h \| \nabla^2 v \|_{L^p(M_K)},
& \quad \forall v
& \in W^{2,p}(M_K),
\end{alignat}
where each $C$ is independent of $h$, $K$, and $v$.
\end{enumerate}
\end{lem}
%

\subsection{Tubular neighborhood}
\label{subsec:tubular}

In order to address the integrals over the boundary-skin $\Omega \bigtriangleup \Omega_h$,
we introduce the tubular neighborhood of $\partial\Omega$.
If $h$ is sufficiently small, we can construct a homeomorphism $\pi \colon \partial\Omega_h \to \partial\Omega$
based on the signed distance function with respect to $\partial\Omega$.
Then, the inverse map $\pi^* \colon \partial\Omega \to \partial\Omega_h$ is of the form
$\pi^*(x) = x + t^*(x)n(x)$ ($x \in \partial\Omega$),
where $n(x)$ is the outward unit normal vector of $\partial\Omega$ at $x$ and $t^* \in C^0(\partial\Omega; \bR)$.
We refer the reader to \cite[Section~14.6]{GilT01} for construction and properties of $\pi$.
It is known that $\| t^* \|_{L^\infty(\partial\Omega)} \le c_0 h^2$ for some $c_0 > 0$ depending only on $\Omega$.
In what follows, we set $\varepsilon := c_0 h^2$ for such $c_0$.
Then, from this observation, we have $\Omega \bigtriangleup \Omega_h \subset T(\varepsilon)$,
where $T(\varepsilon)$ is the tubular neighborhood of $\partial\Omega$ defined by \eqref{eq:tubular}.

Here, we collect some estimates related to $T(\varepsilon)$.
For the proofs of the following inequalities, we refer to \cite[Appendix]{KasOZ16} and \cite[Appendix~A]{KasK17}.
\begin{lem}
\label{lem:tube}
\begin{enumerate}[label=\upshape (\roman{*})]
\item For $f \in L^1(T(\varepsilon))$, we have
\begin{equation}
\left| \int_{\partial\Omega} f ds - \int_{\partial\Omega_h} f \circ \pi ds \right|
 \le C \varepsilon \| f \|_{L^1(\partial\Omega)}.
\label{eq:tube-1}
\end{equation}
\item For $f \in W^{1,p}(T(\varepsilon))$ and $p \in [1,\infty]$, we have
\begin{align}
\| f - f \circ \pi \|_{L^p(\partial\Omega_h)}
&\le C \varepsilon^{1-\frac{1}{p}} \| \nabla f \|_{L^p(T(\varepsilon))},
\label{eq:tube-2}\\
\| f \|_{L^p(T(\varepsilon))}
&\le C \varepsilon^{1/p} \| f \|_{L^p(\partial\Omega)} + C \varepsilon \| \nabla f \|_{L^p(T(\varepsilon))},
\label{eq:tube-3}\\
\| f \|_{L^p(\OmghOmg)}
&\le C \varepsilon^{1/p} \| f \|_{L^p(\partial\Omega_h)} + C \varepsilon \| \nabla f \|_{L^p(\OmghOmg)},
\label{eq:tube-4}
\end{align}
and the local estimate
\begin{align}
\| f - f \circ \pi \|_{L^p(\partial\Omega_h \cap D)}
& \le C \varepsilon^{1-\frac{1}{p}} \| \nabla f \|_{L^p(T(\varepsilon) \cap D_\varepsilon)},
\label{eq:tube-2-loc} \\
\| f \|_{L^p((\OmghOmg) \cap D)}
&\le C \varepsilon^{1/p} \| f \|_{L^p(\partial\Omega_h \cap D_\varepsilon)} + C \varepsilon \| \nabla f \|_{L^p((\OmghOmg) \cap D_\varepsilon)},
\label{eq:tube-4-loc}
\end{align}
for $D \subset \bR^N$ and $D_d := \{ x \in \bR^N \mid \dist(x,D) < d \}$ for $d>0$.
\item Letting $n_h$ be the outward unit normal vector of $\partial\Omega_h$, we have
\begin{equation}
\| n_h - n \circ \pi \|_{L^\infty(\partial\Omega_h)} \le C h.
\label{eq:tube-5}
\end{equation}
\end{enumerate}
Here, each $C$ is independent of $h$ and $f$.
\end{lem}

Let $\tilde{\Omega} := \Omega \cup T(\varepsilon) = \Omegah \cup T(\varepsilon)$.
As mentioned above, $\tilde{w}$ denotes an extension of a given function $w$ defined over $\Omega$ in the sense of Sobolev spaces.
Such extension can be constructed by reflection and is well-defined as a function over $\tilde{\Omega}$.
We can check the following global and local stability of the extension operators (see \cite{AdaF03}):
\begin{align}
\| \tilde{w} \|_{W^{s,p}(\tilde{\Omega})} &\le C \|w\|_{W^{s,p}(\Omega)},
\label{eq:ext-global}\\
\| \tilde{w} \|_{W^{s,p}(T(\varepsilon))} &\le C \|w\|_{W^{s,p}(\Omega \cap T(\varepsilon))},
\label{eq:ext-tube}\\
\| \tilde{w} \|_{W^{s,p}(D \cap T(\varepsilon))} &\le C \|w\|_{W^{s,p}(\Omega \cap D_{2\varepsilon} \cap T(\varepsilon))},
\quad D \subset \bR^N,
\label{eq:ext-local}
\end{align}
for $w \in W^{s,p}(\tilde{\Omega})$, where $C$ depends only on $s$, $p$, and $\Omega$.

\subsection{Regularized delta and Green's functions}

We present preliminary estimates for $\bdelta$, $\Gamma$, and $\Gamma_h$ introduced in Section~\ref{sec:outline} (see~\eqref{eq:delta}, \eqref{eq:gamma}, and \eqref{eq:disc-gamma}).
The regularized delta function $\bdelta$ satisfies $\supp \bdelta \subset \Omega \cap \Omega_h$ (i.e., $\supp \bdelta \cap T(\varepsilon) = \emptyset$) and
\begin{equation}
\| \bdelta \|_{W^{s,p}(K_0)} \le C_{s,p} h^{-s-\left( 1 - \frac{1}{p} \right)N},
\quad \forall s \ge 0, \quad \forall p \in [1,\infty],
\label{eq:norm-delta}
\end{equation}
where $C_{s,p}$ is independent of $h$ and $x_0$ by construction (see~\cite[Appendix]{SchSW96}).
Further, we have
\begin{equation}
| (P_h \bdelta)(x) | \le C h^{-N} e^{-c|x_0-x|/h}, \quad \forall x \in \Omega_h,
\label{eq:deltah}
\end{equation}
where $C$ and $c$ are independent of $h$, $x_0$, and $x$.
The proofs can be found in \cite[Lemma~7.2]{Wah91}.

%

We recall the pointwise esitmates for the usual Green's function (fundamental solution).
Let $G = G(x,y;t)$ be the solution of
\begin{equation}
\begin{cases}
\partial_t G + A G = 0, & \text{in } Q_T, \\
\partial_n G = 0, & \text{on } \partial\Omega \times (0,T), \\
G(0) = \delta_y, & \text{in } \Omega.
\end{cases}
\label{eq:green}
\end{equation}
where $y \in \Omega$ and $\delta_y$ is the Dirac $\delta$-function with respect to $y$.
Then, the following pointwise estimates are known.
\begin{equation}
|\partial_t^k \partial_x^\alpha G(x,y;t)|
\le C \left( \sqrt{t} + |x-y| \right)^{-N-2k-|\alpha|} e^{-c|x-y|^2/t},
\quad \forall x,y \in \Omega, \quad \forall t > 0,
\label{eq:gaussian}
\end{equation}
for any non-negative integer $k$ and multi-index $\alpha$,
where $C$ and $c$ are independent of $x$, $y$, and $t$.
See \cite{EidI70} for the proof.

Since the regularized Green's function $\Gamma$ solves \eqref{eq:gamma}, it can be written as
\begin{equation}
\Gamma(x,t) = \int_\Omega G(x,y;t) \bdelta(y) dy
\end{equation}
for $x \in \Omega$.
This representation gives the following estimates, which is used repeatedly.
Recall that $Q_{h,j}$ is the parabolic dyadic decomposition defined by \eqref{eq:dyadic2}.
\begin{lem}\label{lem:gap}
Let $T \le 1$, $p \in [1,\infty]$, $l \in \mathbb{N}_0$, and $\alpha \in \mathbb{N}_0^N$.
Then, we have
\begin{align}
\| \partial_t^l \partial_x^\alpha \tilde{\Gamma} \|_{L^p(L_T(\varepsilon) \cap Q_{h,j})}
&\le C h^{\frac{2}{p}} d_j^{\frac{1}{p} - \left( 1-\frac{1}{p} \right)N - |\alpha| - 2l},
\label{eq:gap-1}
\\
\| \partial_t^l \partial_x^\alpha \tilde{\Gamma} \|_{L^p(\Sigma_{h,T} \cap Q_{h,j})}
&\le C d_j^{\frac{1}{p} - \left( 1-\frac{1}{p} \right)N - |\alpha| - 2l}.
\label{eq:gap-2}
\end{align}
Moreover, the same estimates hold on $Q_{h,*}$ with $d_j$ replaced by $d_{J_*}$.
\end{lem}

\begin{proof}
We show the first inequality \eqref{eq:gap-1} for $Q_{h,j}$.
By the H\"older inequality and the local stability of the extension \eqref{eq:ext-local}, we have
\begin{equation}
\| \partial_t^l \partial_x^\alpha \tilde{\Gamma} \|_{L^p(L_T(\varepsilon) \cap Q_{h,j})}
\le C (\varepsilon d_j^{N+1})^{1/p}
\sum_{|\beta| \le |\alpha|}
\| \partial_t^l \partial_x^\beta \Gamma \|_{L^\infty(L_T(\varepsilon) \cap Q'_{h,j})}.
\end{equation}
Since $\partial_t^l \partial_x^\beta\Gamma$ is represented as
\begin{equation}
\partial_t^l \partial_x^\beta \Gamma(x,t) = \int_{\supp \bdelta} \partial_t^l \partial_x^\beta G(x,y;t) \bdelta(y) dy,
\end{equation}
we obtain
\begin{equation}
\| \partial_t^l \partial_x^\beta \Gamma \|_{L^\infty(L_T(\varepsilon) \cap Q'_{h,j})}
\le C d_j^{-N-2l-|\beta|}
\le C d_j^{-N-2l-|\alpha|}
\end{equation}
for $|\beta| \le |\alpha|$, from \eqref{eq:gaussian} and $\supp \bdelta \cap T(\varepsilon) = \emptyset$.
Noting that $\varepsilon \approx h^2$, we can derive \eqref{eq:gap-1}.
The proof of \eqref{eq:gap-2} is similar since
\begin{equation}
\| \partial_t^l \partial_x^\alpha \tilde{\Gamma} \|_{L^p(\Sigma_{h,T} \cap Q_{h,j})}
\le C d_j^{(N+1)/p}
\sum_{|\beta| \le |\alpha|}
\| \partial_t^l \partial_x^\beta \Gamma \|_{L^\infty(L_T(\varepsilon) \cap Q_{h,j})},
\end{equation}
holds.
Hence we can complete the proof.
\end{proof}

The first application of the above Lemma is several estimates for $\Gamma$.
\begin{lem}\label{lem:gamma}
Let $T \le 1$, $p \in [1,\infty]$, $l \in \mathbb{N}_0$, and $\alpha \in \mathbb{N}_0^N$.
\begin{enumerate}[label={(\roman*)},font=\upshape]
\item 
Assume that 
\begin{equation}
-\frac{1}{p}+\left( 1 - \frac{1}{p} \right)N+|\alpha|+2l > 0.
\label{eq:assum-gamma}
\end{equation}
Then, we have
\begin{align}
\| \partial_t^l \partial_x^\alpha \tGamma \|_{L^p(L_T(\varepsilon))}
&\le C h^{\frac{3}{p} - \left( 1-\frac{1}{p} \right)N - |\alpha| - 2l}, \\
\| \partial_t^l \partial_x^\alpha \tGamma \|_{L^p(\Sigma_{h,T})}
&\le C h^{\frac{1}{p} - \left( 1-\frac{1}{p} \right)N - |\alpha| - 2l}.
\end{align}
\item If 
\begin{equation}
-\frac{1}{p}+\left( 1 - \frac{1}{p} \right)N+|\alpha|+2l = 0,
\label{eq:assum-gamma-2}
\end{equation}
then we have
\begin{align}
\| \partial_t^l \partial_x^\alpha \tGamma \|_{L^p(L_T(\varepsilon))}
&\le C h^{2/p} |\log h|^{1/p}, \\
\| \partial_t^l \partial_x^\alpha \tGamma \|_{L^p(\Sigma_{h,T})}
&\le C |\log h|^{1/p}.
\end{align}
\item If 
\begin{equation}
-\frac{1}{p}+\left( 1 - \frac{1}{p} \right)N+|\alpha|+2l < 0,
\label{eq:assum-gamma-3}
\end{equation}
then we have
\begin{align}
\| \partial_t^l \partial_x^\alpha \tGamma \|_{L^p(L_T(\varepsilon))}
&\le C h^{2/p}, \\
\| \partial_t^l \partial_x^\alpha \tGamma \|_{L^p(\Sigma_{h,T})}
&\le C .
\end{align}
\end{enumerate}
\end{lem}

In the proof below, and thereafter, we write $\sum_{j,*}$ when the summation includes the integration over $Q_{h,*}$.
If it is not included, we denote the summation by $\sum_j$.

\begin{proof}
Let $p < \infty$. Then, from the previous lemma and \eqref{eq:sum}, we have
\begin{align}
\| \partial_t^l \partial_x^\alpha \tGamma \|_{L^p(L_T(\varepsilon))}^p
&\le C \sum_{j,*} h^2 d_j^{1-(p-1)N-(|\alpha|+2l)p} \\
&\le C h^{3-(p-1)N-(|\alpha|+2l)p} \sum_{j,*} \left( \frac{h}{d_j} \right)^{-1+(p-1)N+(|\alpha|+2l)p} \\
&\le C h^{3-(p-1)N-(|\alpha|+2l)p},
\end{align}
when \eqref{eq:assum-gamma} holds.
The other cases can be obtained similarly since
\begin{equation}
\sum_{j,*} 1 = d_{J_*}+2 \le C |\log h|
\qquad \text{and} \qquad
\sum_{j,*} d_j^r \le \sum_{j \ge 0} (2^r)^{-j-1} \le C
\end{equation}
for $r>0$.
\end{proof}

We also mention the global energy estimates for $F = \Gamma_h - \tGamma$.
\begin{lem}\label{lem:energy-F}
There exists a constant $C>0$ independent of $h$ that satisfies
\begin{equation}
\| F \|_{L^2(0,T;H^1(\Omegah))} + h \| F_t \|_{L^2(Q_{h,T})} \le C h^{-N/2}
\label{eq:energy-F}
\end{equation}
for any $T > 0$.
Moreover, we have
\begin{equation}
\| F \|_{L^2(Q_{h,*})} \le C h^{1-\frac{N}{2}}.
\label{eq:L2-F}
\end{equation}
\end{lem}
\begin{proof}
We first show the bound for the $L^2(0,T;H^1(\Omegah))$-norm.
Substituting $v_h = \Gamma_h$ into \eqref{eq:disc-gamma}, we have
\begin{equation}
\frac{1}{2}\frac{d}{dt} \| \Gamma_h \|_{L^2(\Omegah)}^2 + \| \Gamma_h \|_{H^1(\Omegah)}^2 = 0.
\end{equation}
Integrating this equality on the interval $(0,t)$, we obtain
\begin{equation}
\frac{1}{2} \| \Gamma_h(t) \|_{L^2(\Omegah)}^2 + \| \Gamma_h \|_{L^2(0,t; H^1(\Omegah))}^2 
= \frac{1}{2} \| P_h \bdelta \|_{L^2(\Omegah)}^2
\label{eq:tmp0}
\end{equation}
since $\Gamma_h(0) = P_h \bdelta$.
Therefore, \eqref{eq:norm-delta} gives the estimate $\| \Gamma_h \|_{L^2(0,T;H^1(\Omegah))} \le C h^{-N/2}$.
The estimate for $\tGamma$ is derived in the same way and thus we have $\| F \|_{L^2(0,T;H^1(\Omegah))} \le C h^{-N/2}$ by the triangle inequality.
The bound for $\| \Gamma_{h,t} \|_{L^2(Q_{h,T})}$ can be obtained by substituting $v_h=\Gamma_{h,t}$ into \eqref{eq:disc-gamma} and the estimate for $\| \tGamma_{t} \|_{L^2(Q_{h,T})}$ is as well.
Hence we can derive \eqref{eq:energy-F}.

We show the second inequality \eqref{eq:L2-F}.
Integrating \eqref{eq:tmp0} again, we have
\begin{equation}
\| \Gamma_h \|_{L^2(Q_{h,*})}^2
\le \int_0^{d_{J_*}^2} \| P_h \bdelta \|_{L^2(\Omegah)}^2 dt
\le C h^{2-N},
\end{equation}
which gives $\| \Gamma_h \|_{L^2(Q_{h,*})} \le Ch^{1-\frac{N}{2}}$.
The estimate for $\| \tGamma \|_{L^2(Q_{h,*})}$ is similar and thus we can complete the proof.
\end{proof}

\section{Proof of the main result}
\label{sec:proof}

\subsection{Reduction of the error estimates}
\label{subsec:reduction}

According to \cite{SchTW98}, we reduce the error estimate \eqref{eq:best} to the $L^1$-error estimates for $\Gamma$ and $\Gamma_h$
for $T \le 1$.
In the argument of \cite{SchTW98}, the Galerkin orthogonality (or compatibility)
\begin{equation} 
((u-u_h)_t, v_h)_\Omega + a_\Omega(u-u_h, v_h) = 0,
\quad \forall v_h \in V_h
\end{equation}
holds since $\Omega=\Omega_h$ and this identity is used repeatedly.
However, in our case, there appear additional terms induced by the boundary-skins.
Thus we begin this section by the \emph{asymptotic} Galerkin orthogonality.
In what follows, $\partial_{n_h}$ denotes the outward normal derivative on $\partial\Omegah$.

\begin{lem}[Asymptotic Galerkin orthogonality]
\label{lem:ago}
Assume $z$ solves
\begin{equation}
\begin{cases}
z_t + Az = \varphi, & \text{in } Q_T, \\
\partial_n z = \psi, & \text{on } \partial\Omega \times (0,T),
\end{cases}
\end{equation}
and $z_h$ solves
\begin{equation}
(z_{h,t}, v_h)_{\Omega_h} + a_\Omegah(z_h, v_h)
= (\tilde{\varphi}, v_h)_{\Omega_h} + (\tilde{\psi}, v_h)_{\partial\Omega_h},
\quad \forall v_h \in V_h
\end{equation}
for given $\varphi \in C(\overline{Q_T})$ and $\psi \in C(\partial\Omega \times (0,T))$.
Then, we have
\begin{equation}
((z_h - \tilde{z})_t, v_h)_{\Omega_h} + a_\Omegah(z_h - \tilde{z}, v_h)
=
- (\tilde{z}_t + A \tilde{z} - \tilde{\varphi}, v_h)_\OmghOmg
- (\partial_{n_h} \tilde{z} - \tilde{\psi}, v_h)_{\partial\Omega_h}
\label{eq:ago}
\end{equation}
\end{lem}

\begin{proof}
We observe that the formula
\begin{equation}
(\nabla v, \nabla w)_\OmgOmgh - (\nabla v, \nabla w)_\OmghOmg
= (\partial_n v, w)_{\partial\Omega} - (\partial_{n_h} v, w)_{\partial\Omega_h} 
- (\Delta v, w)_{\OmgOmgh} + (\Delta v, w)_\OmghOmg
\label{eq:ibp}
\end{equation}
holds for $v \in H^2(\Omega \cup \Omegah)$ and $w \in H^1(\Omega \cup \Omegah)$ by integration by parts.
Now, from the identity \eqref{eq:gap}, we have
\begin{equation}
(\tilde{z}_t, v_h)_{\Omega_h} + a_\Omegah(\tilde{z}, v_h)
=
I_1 + I_2,
\end{equation}
where
\begin{equation}
I_1 = (z_t, \tilde{v}_h)_\Omega + a_\Omega(z,\tilde{v}_h)
= (\varphi, \tilde{v}_h)_\Omega + (\psi, \tilde{v}_h)_{\partial\Omega}
\end{equation}
and
\begin{align}
I_2 =
- (z_t, \tilde{v}_h)_{\OmgOmgh} - a_\OmgOmgh(z, \tilde{v}_h)
+ (\tilde{z}_t, v_h)_{\OmghOmg} + a_\OmghOmg(\tilde{z}, v_h).
\end{align}
Again, from \eqref{eq:gap}, we have
\begin{equation}
I_1 = (\tilde{\varphi}, v_h)_{\Omega_h} + (\varphi, \tilde{v}_h)_\OmgOmgh - (\tilde{\varphi}, v_h)_\OmghOmg + (\psi, \tilde{v}_h)_{\partial\Omega}.
\end{equation}
Moreover, due to the formula \eqref{eq:ibp}, we have
\begin{align}
I_2
&= -(z_t + Az, \tilde{v}_h)_\OmgOmgh + (\tilde{z}_t + A \tilde{z}, v_h)_\OmghOmg
- (\partial_n z, \tilde{v}_h)_{\partial\Omega} + (\partial_n \tilde{z}, v_h)_{\partial\Omega_h} \\
&= -(\varphi, \tilde{v}_h)_\OmgOmgh + (\tilde{z}_t + A \tilde{z}, v_h)_\OmghOmg
- (\psi, \tilde{v}_h)_{\partial\Omega} + (\partial_n \tilde{z}, v_h)_{\partial\Omega_h}
\end{align}
Therefore, we obtain
\begin{align}
&(\tilde{z}_t, v_h)_{\Omega_h} + a_\Omegah(\tilde{z}, v_h) \\
&=
(\tilde{\varphi}, v_h)_{\Omega_h}
+ (\tilde{z}_t + A \tilde{z} - \tilde{\varphi}, v_h)_\OmghOmg
+ (\partial_n \tilde{z}, v_h)_{\partial\Omega_h} \\
&=
(\tilde{\varphi}, v_h)_{\Omega_h} + (\tilde{\psi}, v_h)_{\partial\Omega_h} + (\tilde{z}_t + A \tilde{z} - \tilde{\varphi}, v_h)_\OmghOmg
+ (\partial_n \tilde{z} - \tilde{\psi}, v_h)_{\partial\Omega_h},
\end{align}
which implies the desired equality owing to the definition of $z_h$.
\end{proof}

Now, we turn to the error estimates. Assume $T \le 1$ and $u_h(0) = P_h \tilde{u}_0$.
As observed in Section~\ref{sec:outline}, we have
\begin{equation}
|(\tilde{u}-u_h)(x_0,T)|
\le C \inf_{\chi \in C^0([0,T]; V_h)} \|\tilde{u}-\chi\|_{L^\infty(Q_{h,T})} + |((P_h\tilde{u}-u_h)(T), \bdelta)_{\Omega_h}|.
\end{equation}
We address
the last term $((P_h\tilde{u}-u_h)(T), \bdelta)_{\Omega_h}$ that is represented as follows.
Recall that $F = \Gamma_h - \tilde{\Gamma}$.

\begin{lem}\label{lem:reduce}
Assume that $u_{h,0} = P_h \tu_0$. Then,
for any $\chi \in C^0([0,T]; V_h)$, we have
\begin{equation}
((u_h - P_h\tilde{u})(T), \bdelta)_{\Omega_h} = \sum_{j=0}^7 E_j,
\label{eq:expression}
\end{equation}
where
\begin{align}
E_0 &= \int_0^T \left[ (\tilde{u} - \chi, F_t )_{\Omega_h}
+ a_{\Omega_h}(\tilde{u} - \chi, F) \right] dt,
& & \\
E_1 &= \int_0^T (\tilde{u} - \chi, \tilde{\Gamma}_t + A\tilde{\Gamma})_{\Omega_h \setminus \Omega} dt,
&
E_2 &= \int_0^T (\tilde{u} - \chi, \partial_{n_h} \tilde{\Gamma})_{\partial\Omega_h} dt,
\\
E_3 &= \int_0^T (\tilde{f} - \tilde{u}_t - A\tilde{u}, F)_{\Omega_h \setminus \Omega} dt,
&
E_4 &= \int_0^T (\tilde{g} - \partial_{n_h} \tilde{u}, F)_{\partial\Omega_h} dt,
\\
E_5 &= \int_0^T \left[ (\tilde{f} - \tilde{u}_t, \tilde{\Gamma})_{\Omega_h \setminus \Omega} - (f-u_t, \Gamma)_{\Omega \setminus \Omega_h} \right] dt,
& & \\
E_6 &= \int_0^T \left[ a_{\Omega\setminus\Omega_h}(u,\Gamma) - a_{\Omega_h\setminus\Omega}(\tilde{u}, \tilde{\Gamma}) \right] dt
&
E_7 &= \int_0^T \left[ (\tilde{g}, \tilde{\Gamma})_{\partial\Omega_h} - (g,\Gamma)_{\partial\Omega} \right] dt.
\end{align}
Here, in each inner-product, the time of the left function is $t$ and the right is $T-t$.
For example, the term $(g,\Gamma)_{\partial\Omega}$ denotes $(g(t),\Gamma(T-t))_{\partial\Omega}$.
\end{lem}

\begin{proof}
By the asymptotic Galerkin orthogonality \eqref{eq:ago}, we have
\begin{align}
\frac{d}{dt} (u_h(t), \Gamma_h(T-t))_\Omegah
&=
(u_{h,t}(t), \Gamma_h(T-t))_\Omegah + a_\Omegah(u_h(t), \Gamma_h(T-t)) \\
&=
(\tilde{u}_t(t), \Gamma_h(T-t))_\Omegah + a_\Omegah(\tilde{u}(t), \Gamma_h(T-t)) \\
&- (\tilde{u}_t + A \tilde{u} - \tilde{f}, \Gamma_h)_\OmghOmg
- (\partial_\nh \tilde{u} - \tilde{g}, \Gamma_h)_{\partial\Omega_h}.
\end{align}
Here we abbreviated the time variable $t$ and $T-t$ as in the statement of the lemma,
and we use the same abbreviation in the rest of the proof.
Integrating the both sides, we have
\begin{multline}
((u_h - P_h\tilde{u})(T), \bdelta)_{\Omega_h}
= \int_0^T \left[ (\tilde{u}, \Gamma_{h,t})_\Omegah + a_\Omegah(\tilde{u},\Gamma_h) \right] dt \\
- \int_0^T \left[ (\tilde{u}_t + A \tilde{u} - \tilde{f}, \Gamma_h)_\OmghOmg
+ (\partial_\nh \tilde{u} - \tilde{g}, \Gamma_h)_{\partial\Omega_h} \right] dt.
\label{eq:expression-1}
\end{multline}
By the definition of $\Gamma$ and the identity \eqref{eq:gap}, we have
\begin{equation}
(\tilde{u},\tilde{\Gamma}_t)_\Omegah + a_\Omegah(\tilde{u}, \tilde{\Gamma})
+ (u,\Gamma_t)_\OmgOmgh + a_\OmgOmgh(u,\Gamma)
- (\tilde{u},\tilde{\Gamma}_t)_\OmghOmg - a_\OmghOmg(\tilde{u}, \tilde{\Gamma})
= 0,
\end{equation}
which, together with \eqref{eq:ibp}, yields
\begin{equation}
- (\tu, \tGamma_t)_\Omegah - a_\Omegah(\tu, \tGamma)
+ (\tu, \tGamma_t + A \tGamma)_\OmghOmg
+ (\tu, \partial_\nh \tGamma)_{\partial\Omega} = 0.
\label{eq:expression-2}
\end{equation}
Moreover, the asymptotic Galerkin orthogonality \eqref{eq:ago} implies
\begin{equation}
- (\chi, F_t)_\Omegah - a_\Omegah(\chi,F)
- (\chi,\tGamma_t + A \tGamma)_\OmghOmg - (\chi, \partial_\nh \tGamma)_{\partial\Omegah}
= 0.
\label{eq:expression-3}
\end{equation}
for any $\chi \in V_h$, since $\partial_n \Gamma = 0$ on $\partial\Omega$.
Adding \eqref{eq:expression-2} and \eqref{eq:expression-3} to the right hand side of \eqref{eq:expression-1},
we have
\begin{multline}
((u_h - P_h\tilde{u})(T), \bdelta)_{\Omega_h}
= E_0 + E_1 + E_2  \\
+ \int_0^T \left[ (\tf - \tu - A \tu, \Gamma_h)_\OmghOmg
+ (\tg - \partial_\nh \tu, \Gamma_h)_{\partial\Omega_h} \right] dt.
\label{eq:expression-4}
\end{multline}
We calculate $(\tu_t+A\tu-\tf,\tGamma)_\OmghOmg$.
Owing to \eqref{eq:ibp}, we have
\begin{multline}
(\tu_t+A\tu-\tf,\tGamma)_\OmghOmg - (u+Au-f,\Gamma)_\OmgOmgh \\
= (\tu_t - \tf, \tGamma)_\OmghOmg - (u_t-f, \Gamma)_\OmgOmgh
+ a_\OmghOmg(\tu,\tGamma) - a_\OmgOmgh(u,\Gamma)
+ (\partial_\nh \tu, \tGamma)_{\partial\Omegah} - (g,\Gamma)_{\partial\Omega},
\end{multline}
which implies
\begin{multline}
(\tu_t+A\tu-\tf,\tGamma)_\OmghOmg
= \left[ (\tu_t - \tf, \tGamma)_\OmghOmg - (u_t-f, \Gamma)_\OmgOmgh  \right]
+ \left[ a_\OmghOmg(\tu,\tGamma) - a_\OmgOmgh(u,\Gamma) \right] \\
+ \left[ (\tg, \tGamma)_{\partial\Omegah} - (g,\Gamma)_{\partial\Omega} \right]
+ (\partial_\nh \tu - \tg, \tGamma)_{\partial\Omegah}
\label{eq:expression-5}
\end{multline}
since $(u+Au-f,\Gamma)_\OmgOmgh = 0$.
We can obtain the desired equation from \eqref{eq:expression-4} and \eqref{eq:expression-5}.
\end{proof}

In the expression \eqref{eq:expression}, the principal part is $E_0$, since other terms, which are induced by the boundary-skin, disappear when $\Omega = \Omegah$.
We can address the term $E_0$ in the same way as \cite[Section~3]{SchTW98}, 
since the calculation is performed on the domain $\Omegah$ only.
Indeed, we can obtain the following estimate with the aid of \eqref{eq:ext-global} and \eqref{eq:ext-local}:
\begin{equation}
|E_0| \le C \left( |\log h|^{\underline{k}} + \| F_t \|_{L^1(Q_{h,T})} + h^{-1} |\log h|^{-\underline{k}} \| F \|_{L^1(0,T;W^{1,1}(\Omegah))} \right)
\| \tilde{u} - \chi \|_{L^\infty(Q_{h,T})}.
\label{eq:E0}
\end{equation}
In order to handle other terms, 
we recall the estimates given in \cref{lem:gamma}.

\begin{lem}
\label{lem:Ej}
Assume $T \le 1$.
Let $E_j \, (j=1,2,\dots,7)$ be the terms appearing in \cref{lem:reduce}.
Then, we have
\begin{align}
\sum_{j=1}^{2} |E_j| &\le C \| \tu - \chi \|_{L^\infty(Q_{h,T})},
\label{eq:E12} \\
\sum_{j=3}^{5} |E_j| &\le C h^2 \left( 1 + h^{-1} \| F \|_{L^1(0,T;W^{1,1}(\Omegah))} \right)  \left( \| u \|_{L^\infty(0,T;W^{2,\infty}(\Omega))} + \| u_t \|_{L^\infty(Q_T)} \right),
\label{eq:E345} \\
\sum_{j=6}^{7} |E_j| &\le Ch^2 |\log h| \| u \|_{L^\infty(0,T;W^{2,\infty}(\Omega))},
\label{eq:E67}
\end{align}
for any $\chi \in C^0([0,T]; V_h)$.
\end{lem}

\begin{proof}
Since \cref{lem:gamma} yields
\begin{equation}
\| \tGamma_t + A \tGamma \|_{L^1(L_T(\varepsilon))}
\le C h,
\label{eq:E1-1}
\end{equation}
we have
\begin{equation}
|E_1| \le C h \| \tu- \chi \|_{L^\infty(Q_{h,T})}.
\end{equation}
For the estimate of $E_2$, we recall \cref{lem:tube}.
Noting that $\nabla \Gamma \cdot n = 0$ on $\partial\Omega$, we have
\begin{equation}
\| \partial_\nh \tGamma \|_{L^1(\partial \Omegah)}
\le \| \nabla\tGamma \cdot (n_h - n \circ \pi) \|_{L^1(\partial\Omegah)}
+ \| [ \nabla\tGamma - (\nabla\Gamma \circ \pi) ] \cdot (n \circ \pi) \|_{L^1(\partial\Omegah)}
\label{eq:tmp1}
\end{equation}
at each time.
Therefore, owing to \eqref{eq:tube-5}, \eqref{eq:tube-2}, and \cref{lem:gamma}, we have
\begin{align}
\| \partial_\nh \tGamma \|_{L^1(\Sigma_{h,T})}
&\le C h \| \nabla \tGamma \|_{L^1(\Sigma_{h,T})}
+ C \| \nabla^2 \tGamma \|_{L^1(L_T(\varepsilon))} \\
&\le C h |\log h|,
\label{eq:E2-1}
\end{align}
which leads to
\begin{equation}
|E_2| \le Ch|\log h| \| \tu-\chi\|_{L^\infty(Q_{h,T})}.
\end{equation}
Hence we establish \eqref{eq:E12}.

Let us prove \eqref{eq:E345}.
By \eqref{eq:tube-3} and the trace inequality, we have
\begin{equation}
|E_3| \le C \left( \| u_t \|_{L^\infty(Q_T)} + \| u \|_{L^\infty(0,T;W^{2,\infty}(\Omega))} \right)
\cdot \varepsilon \| F \|_{L^1(0,T; W^{1,1}(\Omega))}.
\label{eq:E3}
\end{equation}
Noting that $g \circ \pi = ((\nabla u) \circ \pi) \cdot (n \circ \pi)$,
we have, at each time,
\begin{align}
& \| \tg - \partial_\nh \tu \|_{L^\infty(\partial\Omegah)} \\
&\le \| \tg - g \circ \pi \|_{L^\infty(\partial\Omegah)}
+ \| ((\nabla u) \circ \pi - \nabla \tu) \cdot (n \circ \pi) \|_{L^\infty(\partial\Omegah)}
+ \| \nabla \tu \cdot (n \circ \pi - \nh) \|_{L^\infty(\partial\Omegah)} \\
&\le Ch \| u \|_{L^\infty(0,T; W^{2,\infty}(\Omega))}
\end{align}
from \eqref{eq:tube-2} and \eqref{eq:tube-5}.
Therefore, we have
\begin{align}
|E_4|
&\le
\| \tg - \partial_\nh \tu \|_{L^\infty(\Sigma_{h,T})}
\| F \|_{L^1(\Sigma_{h,T})} \\
&\le
Ch^2 \| u \|_{L^\infty(0,T; W^{2,\infty}(\Omega))}
\cdot h^{-1} \| F \|_{L^1(0,T;W^{1,1}(\Omega))} .
\label{eq:E4}
\end{align}
The estimate of $E_5$ follows from \cref{lem:gamma}.
Indeed, since $\| \tGamma \|_{L^1(L_T(\varepsilon))} \le C h^2$, we have
\begin{align}
|E_5|
&\le
C \| \tf - \tu_t \|_{L^\infty(L_T(\varepsilon))}
\cdot \| \tGamma \|_{L^1(L_T(\varepsilon))} \\
&\le
C h^2 \left( \| u_t \|_{L^\infty(Q_T)} + \| u \|_{L^\infty(0,T;W^{2,\infty}(\Omega))} \right).
\label{eq:E5}
\end{align}
Summarizing \eqref{eq:E3}, \eqref{eq:E4}, and \eqref{eq:E5}, we can obtain \eqref{eq:E345}.

Finally, we show \eqref{eq:E67}.
We first observe that the inequality $\| \nabla \tGamma \|_{L^1(L_T(\varepsilon))} \le C h^2 |\log h|$
holds from \cref{lem:gamma}.
Thus we have
\begin{equation}
|E_6| \le \| \nabla \tu \|_{L^\infty(L_T(\varepsilon))}
\cdot \| \nabla \tGamma \|_{L^1(L_T(\varepsilon))}
\le C h^2 |\log h| \| u \|_{L^\infty(0,T; W^{2,\infty}(\Omega))}.
\label{eq:E6}
\end{equation}
To address $E_7$, we perform the calculation similar to \eqref{eq:tmp1}.
Observe that
\begin{align}
|(\tg, \tGamma)_{\partial\Omegah} - (g,\Gamma)_{\partial\Omega}|
&\le
\int_{\partial\Omegah} |\tg\tGamma - (g\Gamma) \circ \pi | ds
+ \left| \int_{\partial\Omegah} (g\Gamma) \circ \pi ds
- \int_{\partial\Omega} g\Gamma ds \right| \\
&\le
C \| \nabla(\tg \tGamma) \|_{L^1(T(\varepsilon))}
+ C \varepsilon \| g\Gamma \|_{L^1(\partial\Omega)} \\
&\le
C \| u \|_{W^{2,\infty}(\Omega)} \left( \| \tGamma \|_{L^1(T(\varepsilon))} + \| \nabla \tGamma \|_{L^1(T(\varepsilon))} + \varepsilon \|\tGamma \|_{L^1(\partial\Omega)} \right)
\label{eq:tmp2}
\end{align}
from \eqref{eq:tube-1} and \eqref{eq:tube-2}.
By \cref{lem:gamma}, we have $\|\tGamma \|_{L^1(\partial\Omega)} \le C$, 
and, as observed, $\| \tGamma \|_{L^1(L_T(\varepsilon))} \le C h^2$
and
$\| \nabla \tGamma \|_{L^1(L_T(\varepsilon))} \le C h^2 |\log h|$ hold.
Together with \eqref{eq:tmp2}, they yield
\begin{equation}
|E_7| \le C h^2 |\log h| \| u \|_{L^\infty(0,T; W^{2,\infty}(\Omega))}.
\label{eq:E7}
\end{equation}
The desired estimate \eqref{eq:E67} follows from \eqref{eq:E6} and \eqref{eq:E7} immediately, and thus we can complete the proof of \cref{lem:Ej}.
\end{proof}

\begin{proof}[Proof of \cref{lem:pointwise}]
Now, we are in a position to show \cref{lem:pointwise}.
Substituting \eqref{eq:E0} and the results of \cref{lem:Ej} into \eqref{eq:expression}, and going back to \eqref{eq:pointwise-0},
we can obtain the desired estimate \eqref{eq:pointwise}.
\end{proof}


\subsection{Proof of \cref{lem:green}}
\label{subsec:proof}

In this subsection, we admit that \cref{lem:local-F,lem:duality} hold for now 
and complete the proof of \cref{lem:green}.
The proofs of \cref{lem:local-F,lem:duality} will be given in Sections~\ref{sec:local} and~\ref{sec:dual}, respectively.

\begin{proof}[Proof of \cref{lem:green}.]
By definition of $Q_{h,j}$ and the H\"older inequality, we have
\begin{equation}
\| F \|_{L^1(0,T; W^{1,1}(\Omegah))}
= \sum_{j,*} \left( \| F \|_{L^1(Q_{h,j})} + \| \nabla F \|_{L^1(Q_{h,j})} \right)
\le C \sum_{j,*} d_j^{\frac{N}{2}+1} \trplnorm{F}{1,Q_{h,j}}
\end{equation}
and \cref{lem:energy-F} implies
\begin{equation}
\| F \|_{L^1(0,T; W^{1,1}(\Omegah))} 
\le Ch + C \sum_j d_j^{\frac{N}{2}+1} \trplnorm{F}{1,Q_{h,j}}
\label{eq:green-3}
\end{equation}
since $d_{J_*} \approx C_* h$.
Similarly,
\begin{equation}
\| F_t \|_{L^1(Q_{h,T})} 
\le C + C \sum_j d_j^{\frac{N}{2}+1} \trplnorm{F_t}{Q_{h,j}}.
\label{eq:green-4}
\end{equation}
Therefore, in order to bound $\| F \|_{L^1(0,T; W^{1,1}(\Omegah))}$, we multiply \eqref{eq:local-F} by $d_j^{\frac{N}{2}+2}$ and then sum up with respect to $j$.
For $\| F_t \|_{L^1(Q_{h,T})}$, we replace $d_j^{\frac{N}{2}+2}$ by $d_j^{\frac{N}{2}+1}$ and repeat the same calculation.

Now, multiplying \eqref{eq:local-F} by $d_j^{\frac{N}{2}+2}$ and summing up, we have
\begin{multline}
\theta \sum_j d_j^{\frac{N}{2}+2} \trplnorm{F_t}{Q_{h,j}} 
+ \sum_j d_j^{\frac{N}{2}+1} \trplnorm{F}{1,Q_{h,j}}
+ \theta \sum_j \lambda_j d_j^{\frac{N}{2}+2} \| F(T) \|_{1, D_{h,j}}
\\
\le C_0 \theta \left( \theta \sum_j d_j^{\frac{N}{2}+2} \trplnorm{F_t}{Q'_{h,j}} 
+ \sum_j d_j^{\frac{N}{2}+1} \trplnorm{F}{1,Q'_{h,j}}
+ \theta \sum_j \lambda'_j d_j^{\frac{N}{2}+2} \| F(T) \|_{1, D'_{h,j}} 
\right) \\
+ C \sum_j d_j^{\frac{N}{2}+2} (I_j+X_j+G_j) 
+ C \sum_j d_j^{\frac{N}{2}} \trplnorm{F}{Q'_{h,j}}.
\end{multline}
Recall that the summation for $Q'_{h,j}$ is rewritten in terms of $Q_{h,j}$ (see \eqref{eq:summation-Qhj-prime}).
Then, together with \cref{lem:energy-F}, we have
\begin{multline}
\sum_j \left( d_j^{\frac{N}{2}+2} \trplnorm{F_t}{Q'_{h,j}} + d_j^{\frac{N}{2}+1} \trplnorm{F}{1,Q'_{h,j}} \right) \\
\le C h
+ 3 \cdot 2^{\frac{N}{2}+2} \sum_j \left( d_j^{\frac{N}{2}+2} \trplnorm{F_t}{Q_{h,j}} + d_j^{\frac{N}{2}+1} \trplnorm{F}{1,Q_{h,j}} \right),
\label{eq:sum-1}
\end{multline}
and similarly
\begin{equation}
\sum_j d_j^{\frac{N}{2}} \trplnorm{F}{Q'_{h,j}}
\le C h + 3 \cdot 2^{\frac{N}{2}} \sum_j d_j^{\frac{N}{2}} \trplnorm{F}{Q_{h,j}}.
\label{eq:sum-2}
\end{equation}
Moreover, we can see that
\begin{equation}
\sum_j \lambda'_j d_j^{\frac{N}{2}+2} \| F(T) \|_{1, D'_{h,j}}
\le 3 \cdot 2^{\frac{N}{2}+2} \sum_j \lambda_j d_j^{\frac{N}{2}+2} \| F(T) \|_{1, D_{h,j}},
\label{eq:sum-trace-T}
\end{equation}
which will be proved in Appendix~\ref{sec:appendix}.

Therefore, letting $\theta = C_0^{-1} \cdot 3^{-1} \cdot 2^{-\frac{N}{2}-3}$, 
we can kick-back the terms with local energy norms and the trace at $t=T$.
Consequently, we obtain
\begin{multline}
\sum_j \left( d_j^{\frac{N}{2}+2} \trplnorm{F_t}{Q_{h,j}} + d_j^{\frac{N}{2}+1} \trplnorm{F}{1,Q_{h,j}} \right) \\
\le Ch + C \sum_j d_j^{\frac{N}{2}+2} (I_j+X_j+G_j) 
+ C \sum_j d_j^{\frac{N}{2}} \trplnorm{F}{Q_{h,j}}.
\label{eq:local-F-1}
\end{multline}

The estimates of $I_j$ and $X_j$ are the same as in \cite{SchTW98} and thus we have
\begin{align}
I_j &\le C (h^{-1-N} + h^{-N} d_j^{-1}) d_j^{N/2} e^{-cd_j/h} \le C h^k d_j^{-\frac{N}{2}-k-1}
\label{eq:Ij} \\
X_j &\le C \left( h^{k+1} d_j^{-\frac{N}{2}-k-2} + h^k d_j^{-\frac{N}{2}-k-1} \right) \le C h^k d_j^{-\frac{N}{2}-k-1}
\label{eq:Xj}
\end{align}
from \eqref{eq:deltah} and \eqref{eq:gaussian}.
Moreover, \cref{lem:gap} yields
\begin{equation}
G_j \le C h^2 d_j^{-\frac{N}{2}-2} + C  h d_j^{-\frac{N}{2}-1}.
\end{equation}
Therefore, substituting them into \eqref{eq:local-F-1}, we have
\begin{equation}
\sum_j \left( d_j^{\frac{N}{2}+2} \trplnorm{F_t}{Q_{h,j}} + d_j^{\frac{N}{2}+1} \trplnorm{F}{1,Q_{h,j}} \right) 
\le Ch|\log h|^{\underline{k}} 
+ C \sum_j d_j^{\frac{N}{2}} \trplnorm{F}{Q_{h,j}}
\label{eq:green-5}
\end{equation}
owing to \eqref{eq:sum}.

Now, we apply the local $L^2$-estimate \eqref{eq:dual}.
Multiplying \eqref{eq:dual} by $d_j^{\frac{N}{2}}$ and summing up, we have
\begin{multline}
\sum_j d_j^{\frac{N}{2}} \trplnorm{F}{Q_{h,j}} \\
\le Ch + C \sum_i \left( h^2 \trplnorm{F_t}{Q_{h,i}} + h \trplnorm{F}{1,Q_{h,i}} \right)
\sum_j d_j^{\frac{N}{2}}
\min \left\{ \left( \frac{d_i}{d_j} \right)^{\frac{N}{2}+1}, \left( \frac{d_j}{d_i} \right)^{\frac{N}{2}+1} \right\} \\
+ C C_*^{-1} \left( \sum_j d_j^{\frac{N}{2}} \trplnorm{F}{Q_{h,j}} + \sum_j d_j^{\frac{N}{2}+1} \trplnorm{F}{1,Q_{h,j}} \right)
+ Ch |\log h| \| F \|_{L^1(0,T; W^{1,1}(\Omega_h))},
\end{multline}
owing to \eqref{eq:sum} and $h \le C_*^{-1}d_j$.
Here, we replaced $Q'_{h,j}$ by $Q_{h,j}$ in the summation as in \eqref{eq:sum-1} and \eqref{eq:sum-2}.
Since $\{ d_j \}_j$ is a geometric sequence, we can observe
\begin{equation}
\sum_{j \ge i} d_j^\alpha \le C d_i^\alpha,
\quad
\sum_{j \le i} d_j^{-\alpha} \le C d_j^{-\alpha}
\end{equation}
for $\alpha > 0$. This implies
\begin{equation}
\sum_j d_j^{\frac{N}{2}}
\min \left\{ \left( \frac{d_i}{d_j} \right)^{\frac{N}{2}+1}, \left( \frac{d_j}{d_i} \right)^{\frac{N}{2}+1} \right\}
\le C d_i^{\frac{N}{2}},
\end{equation}
and thus we have
\begin{multline}
\sum_j d_j^{\frac{N}{2}} \trplnorm{F}{Q_{h,j}} 
\le C h
+ C C_*^{-1} \left( \sum_j d_j^{\frac{N}{2}} \trplnorm{F}{Q_{h,j}} + \sum_j d_j^{\frac{N}{2}+2} \trplnorm{F_t}{Q_{h,j}} + \sum_j d_j^{\frac{N}{2}+1} \trplnorm{F}{1,Q_{h,j}} \right) \\
+ C h |\log h| \| F \|_{L^1(0,T; W^{1,1}(\Omega_h))},
\end{multline}
together with $hd_i^{-1} \le C_*^{-1}$, which implies
\begin{multline}
\sum_j d_j^{\frac{N}{2}} \trplnorm{F}{Q_{h,j}} 
\le C h
+ C C_*^{-1} \sum_j \left( d_j^{\frac{N}{2}+2} \trplnorm{F_t}{Q_{h,j}} + d_j^{\frac{N}{2}+1} \trplnorm{F}{1,Q_{h,j}} \right) \\
+ C h |\log h| \| F \|_{L^1(0,T; W^{1,1}(\Omega_h))},
\label{eq:dual-F_t}
\end{multline}
with $C_*$ large enough (independently of $h$).

Substituting \eqref{eq:dual-F_t} into \eqref{eq:green-5}
and again letting $C_*$ large enough to kick-back the summation in \eqref{eq:dual-F_t},
we have
\begin{equation}
\sum_j d_j^{\frac{N}{2}+1} \trplnorm{F}{1,Q_{h,j}}
\le C h|\log h|^{\underline{k}} + C h|\log h| \| F \|_{L^1(0,T; W^{1,1}(\Omega_h))}.
\end{equation}
Going back to \eqref{eq:green-3} and letting $h$ small enough, we establish
\begin{equation}
\| F \|_{L^1(0,T; W^{1,1}(\Omega_h))} \le C h|\log h|^{\underline{k}}.
\end{equation}
Repeating the same argument with \eqref{eq:green-4}, we can achieve
\begin{equation}
\| F_t \|_{L^1(Q_{h,T})} \le C + C \| F \|_{L^1(0,T; W^{1,1}(\Omega_h))} \le C.
\end{equation}
Hence, we complete the proof of \cref{lem:green},
and thus we can obtain the maximum-norm error estimate \eqref{eq:best} for $T \le 1$.
\end{proof}

\subsection{Proof of theorems for $T \ge 1$}
\label{subsec:long-time}

In the rest of this section, we show that \cref{thm:main,thm:semigroup,thm:dmr} for $T \ge 1$ are derived from the corresponding results for $T \le 1$.
We first show the exponentially decaying property for the discrete heat semigroup generated by $A_h$,
which corresponds to \cite[Lemma~3.3]{SchTW98} for the case $\Omega = \Omegah$.

\begin{lem}
Let $s \ge 0$ and $m > N/2$.
Then, we can find $\gamma > 0$ independently of $h$ which satisfies
\begin{equation}
\| A_h^s e^{-tA_h} v_h \|_{L^\infty(\Omegah)} \le C t^{-s-m} e^{-\gamma t} \| v_h \|_{L^\infty(\Omegah)},
\quad \forall v_h \in V_h, \, \forall t > 0,
\label{eq:decay}
\end{equation}
where $C$ is independent of $h$.
\end{lem}

\begin{proof}
We show that
\begin{equation}
\| A_h^{-1} f_h \|_{L^q(\Omegah)} \le C \| f_h \|_{L^p(\Omegah)},
\quad \forall f_h \in V_h,
\label{eq:inv}
\end{equation}
for any $1 < p < q \le \infty$ with $1/p - 1/q < 1/N$, where $C$ is independent of $h$ and $f_h$.
Once we obtain \eqref{eq:inv}, the proof of \eqref{eq:decay} is similar to that of \cite[Lemma~3.3]{SchTW98}.

Fix $f_h \in V_h$ arbitrarily and let $\tilde{f}_h$ be the extension of $f_h$ which vanishes outside of $\Omegah$.
We consider the elliptic equation
\begin{equation}
\begin{cases}
Au = \tf_h, & \text{in } \Omega, \\
\partial_n u = 0, & \text{on } \partial\Omega
\end{cases}
\end{equation}
and its discrete problem
\begin{equation}
a_\Omegah(u_h,v_h) = (f_h, v_h)_\Omegah, \quad \forall v_h \in V_h,
\end{equation}
so that $u_h = A_h^{-1} f_h$.
Note that $u \in W^{2,r}(\Omega)$ for arbitrary $r \in (1,\infty)$.
Then, since $f_h$ can be viewed as an extension of $\tf_h$, we have
\begin{equation}
\| u_h - P_h \tu \|_{W^{1,r}(\Omega_h)} \le C h \| u \|_{W^{2,r}(\Omega)}
\label{eq:err-ell}
\end{equation}
for $r \in [2, \infty]$.
Indeed, \eqref{eq:err-ell} is proved for $r=2$ in \cite[Theorem~3.1]{BarE88} and for $r=\infty$ in \cite[Theorem~3.1]{KasK17}.
Thus, \eqref{eq:err-ell} for general $r \in [2,\infty]$ is derived by interpolation (cf.~\cite{CalM83}).

Now, let $1 < p < q \le \infty$ satisfy $1/p - 1/q < 1/N$.
Then, from the Sobolev embedding $W^{1,p}(\Omega) \hookrightarrow W^{2,q}(\Omega)$, the inverse inequality, the error estimate \eqref{eq:err-ell}, and the elliptic regularity $\| u \|_{W^{2,p}(\Omega)} \le C \|Au\|_{L^p(\Omega)}$, we have
\begin{align}
\| u_h \|_{L^q(\Omegah)}
&\le \| u_h - P_h \tu \|_{W^{1,q}(\Omegah)} + \| P_h \tu \|_{W^{1,q}(\Omegah)} \\
&\le C h^{-N\left( \frac{1}{p} - \frac{1}{q} \right)} \| u_h - P_h \tu \|_{W^{1,p}(\Omegah)}
+ C \| u \|_{W^{2,p}(\Omega)} \\
&\le C \left( h^{1-N\left( \frac{1}{p} - \frac{1}{q} \right)} + 1 \right) \| u \|_{W^{2,p}(\Omega)} \\
&\le C \| Au \|_{L^p(\Omega)} \le C \| f_h \|_{L^p(\Omegah)},
\end{align}
which yields \eqref{eq:inv}.
Hence we can complete the proof.
\end{proof}

\begin{lem}
Assume that \cref{thm:main,thm:semigroup,thm:dmr} hold for $T \le 1$.
Then, they also hold for $T \ge 1$.
\end{lem}

\begin{proof}
Assume that \cref{thm:semigroup} holds for $T \le 1$.
Then, for $p = \infty$, we can extend \cref{thm:semigroup} to the case $T > 1$ together with \eqref{eq:decay}.
Moreover, since $A_h$ is symmetric and positive definite in $L^2(\Omegah)$ uniformly in $h$, 
we can obtain \cref{thm:semigroup} for $p=2$ and $T \ge 1$ by the spectral decomposition.
Therefore, the estimate \eqref{eq:semigroup} for general $p$ is derived from the Riesz-Thorin theorem and the symmetry.

Consequently, the semigroup $e^{-tA_h}$ is analytic and decays exponentially on $L^q(\Omegah)$ for any $q \in (1,\infty)$.
Thus, if \cref{thm:dmr} holds for $T \le 1$, we can show that it holds for any $T > 0$ by a general theory on maximal regularity
 (cf.~\cite[Theorem~2.4]{Dor93}).
 
Also, \cref{thm:main} for $T>1$ follows from \cref{thm:semigroup} and the $L^\infty$-error estimates for stationary problems that is proved in \cite[Theorem~3.1]{KasK17}.
We can proceed the same argument as in \cite[Lemma~3.4]{SchTW98} by replacing $\Omega$ by $\Omegah$, and thus we omit it.
\end{proof}

\section{Local energy error estimates}
\label{sec:local}

In this section, we show \cref{lem:local-F}.
As in \cite{SchTW98}, we derive the result from the local energy error estimates.

\begin{lem}[Local energy error estimate]
\label{lem:local}
Assume that $T \le 1$ and that $\cT_h$ is shape-regular and quasi-uniform.
Let $D \subset \Omegah$, $I = [t_0,t_1] \subset [0,T]$, $Q = D \times I$,
$D_d = \{ x \in \Omegah \mid \dist(x,D) < d \}$,
$I_d = [t_0 - d^2, t_1+d^2] \cap [0,T]$, and $Q_d = D_d \times I_d$ for $d \in (h, \diam \Omega)$.
Assume that $z \in C^0([0,T]; W^{k+1,\infty}(\Omega))$ and $z_h \in C^0([0,T]; V_h)$ satisfy
\begin{equation}
z_t + A z = 0, \text{ in } Q_T, \quad \partial_n z = 0, \text{ on } \partial\Omega \times (0,T),
\end{equation}
and
\begin{equation}
(z_{h,t}, \chi)_{\Omega_h} + a_{\Omega_h}(z_h, \chi) = 0, \quad \forall \chi \in V_h,
\end{equation}
respectively.
Finally, let $e = z_h - \tz$ and $\zeta = \tz - I_h \tz$.

Then, 
there exist $C_0>0$, $C>0$, and $c>0$ independently of $h$, $d$, $D$, and $I$ 
such that $d \ge ch$ implies, for arbitrary $\theta>0$,
\begin{multline}
\theta \trplnorm{e_t}{Q} + d^{-1} \trplnorm{e}{1,Q} + \theta \lambda_d \| e(T) \|_{1,D} \\
\le
  C_0 \theta 
  \left( \theta \trplnorm{e_t}{Q_{d}}
+ d^{-1} \trplnorm{e}{1,Q_{d}}
+ \theta \lambda_d \| e(T) \|_{1,D_{d}}\right)
+ C d^{-2} \trplnorm{e}{Q_{d}} \\
+ C \left( \kappa_d I_{D_d} + X_{Q_{d}} + H_{Q_{d}} + G_{Q_{d}} \right),
\label{eq:local}
\end{multline}
where 
\begin{equation}
\kappa_d = \begin{cases}
1, & t_0 \le d^2, \\ 0, & t_0 > d^2,
\end{cases}
\qquad
\lambda_d = \begin{cases}
1, & t_1+d^2 \ge T, \\ 0, & t_1+d^2 < T,
\end{cases}
\label{eq:kappa-lambda}
\end{equation}
and
\begin{align}
I_{D'} &:= \| e(0) \|_{1, D'} + d^{-1} \| e(0) \|_{D'}, \\
X_{Q'} &:= d \trplnorm{\zeta_t}{1,Q'} + \trplnorm{\zeta_t}{Q'}
+ d^{-1} \trplnorm{\zeta}{1,Q'} + d^{-2} \trplnorm{\zeta}{Q'}, \\
H_{Q'} &:= C_*^{-1/2} \left( \trplnorm{e_t}{Q'} + d^{-1} \trplnorm{e}{1,Q'} \right) , \\
G_{Q'} &:= 
\begin{multlined}[t]
hd^{\frac{3}{2}} \trplnorm{\tz_{tt} + A\tz_t}{L_T(\varepsilon) \cap Q'}
+ hd^{-\frac{3}{2}} \trplnorm{\tz_t + A\tz}{L_T(\varepsilon) \cap Q'} \\
+ d^{\frac{3}{2}} \trplnorm{\partial_{n_h} \tz_t}{\Sigma_{h,T} \cap Q'}
+ d^{-\frac{3}{2}} \trplnorm{\partial_{n_h} \tz}{\Sigma_{h,T} \cap Q'}
\end{multlined}
\end{align}
for $D' \subset \Omegah$ and $Q' \subset Q_{h,T}$.
\end{lem}

We put aside the proof for now and we here show \cref{lem:local-F}.

\begin{proof}[Proof of \cref{lem:local-F}.]
We substitute $z = \Gamma$, $z_h = \Gamma_h$, $d = d_j$, and
\begin{equation}
Q = \Omega_{h,j} \times [0,d_j^2]
\quad \text{or} \quad
Q = \{ x \in \Omegah \mid |x-x_0| < d_j \} \times [d_j^2,4d_j^2]
\end{equation}
into \eqref{eq:local}.
Then, we have
\begin{multline}
\theta \trplnorm{F_t}{Q_{h,j}} + d_j^{-1} \trplnorm{F}{1,Q_{h,j}} + \theta \lambda_j \| F(T) \|_{1,D_{h,j}} \\
\le
  C_0 \theta
  \left( \theta \trplnorm{F_t}{Q'_{h,j}}
+ d_j^{-1} \trplnorm{F}{1,Q'_{h,j}}
+ \theta \lambda'_j \| F(T) \|_{1,D'_{h,j}}\right)
+ C d_j^{-2} \trplnorm{F}{Q'_{h,j}} \\
+ C_1 C_*^{-1/2} \left( \trplnorm{F_t}{Q'_{h,j}} + d_j^{-1} \trplnorm{F}{1,Q'_{h,j}} \right)
+ C \left( I_j + X_j + G_j \right),
\end{multline}
for arbitrary $\theta>0$,
where $I_j$, $X_j$, and $G_j$ are defined in the statement of \cref{lem:local-F}.
Here, we denote the constant before $C_*^{-1/2}$ by $C_1$, which is induced by the term $H_{Q_d}$ in \cref{lem:local} and thus independent of $h$ and $C_*$.
Making $C_*^{-1}$ small enough so that $C_1 C_*^{-1/2} \le C_0 \theta^2$ (recall that $\theta$ is chosen depending only on $C_0$ and $N$ in the proof of \cref{lem:green}) and replacing $2C_0$ by $C_0$, we obtain the desired estimate \eqref{eq:local-F}.
\end{proof}

Now we give the proof of \cref{lem:local}. 
The outline is based on that of \cite[Lemma~6.1]{SchTW98} and \cite[Lemma~4.1]{ThoW00}.
In these proofs, the strong super-approximation property is introduced and proved for Lagrangian finite element spaces.
However, we have succeeded in avoiding these arguments.
Thus we give an alternative proof.

\begin{proof}[Proof of \cref{lem:local}]
We first introduce a cut-off function $\omega$ according to \cite{SchTW98}.
Let $\omega_1 \in C^\infty(\Omega_h)$ satisfy
\begin{equation}
0 \le \omega_1 \le 1, \quad \omega_1|_D \equiv 1, \quad \omega_1|_{\Omega_h \setminus D_d} \equiv 0,
\quad |D^l \omega_1| \le Cd^{-l}
\end{equation}
for $l \in \bN_0$.
We can find such $\omega_1$ if $d \ge 2h$ since $\cT_h$ is quasi-uniform.
We also choose $\omega_2 \in C^1[0,T]$ that satisfies 
\begin{equation}
0 \le \omega_2 \le 1, \quad \omega_2|_{\hat{I}} \equiv 1,  
\quad \supp \omega_2 = I_d,
\quad |\omega_2'| \le Cd^{-2},
\end{equation}
where $\hat{I} = [t_1,T]$ if $t_2+d^2 \ge T$ and $\hat{I} = I$ otherwise.  
We finally set $\omega(x,t) = \omega_1(x) \omega_2(t)$ for $(x,t) \in Q_{h,T}$.

\textit{Step 1.} We first consider the local $L^2$-$H^1$-estimate. 
Let $\zeta_h = z_h - I_h z = e + \zeta$.
Then, by an elementary calculation, we have
\begin{equation}
\frac{1}{2} \frac{d}{dt} \| \omega e \|_\Omegah^2 + \| \omega \nabla e \|_\Omegah^2 + \| \omega e \|_\Omegah^2
= J_1 + J_2,
\end{equation}
where
\begin{align}
J_1 &= (e_t, \omega^2 \zeta_h)_\Omegah + a_\Omegah(e, \omega^2 \zeta_h), \\
J_2 &= -(e_t, \omega^2 \zeta)_\Omegah + (e, \omega \omega_t e)_\Omegah
- a_\Omegah(e, \omega^2 \zeta_h)
+ \| \omega \nabla e \|_\Omegah^2 + \| \omega e \|_\Omegah^2.
\end{align}
We can calculate $J_2$ as
\begin{equation}
J_2 = -(e_t, \omega^2 \zeta)_\Omegah + (e, \omega \omega_t e)_\Omegah
- 2(\nabla e, \omega (\nabla \omega) \zeta_h)_\Omegah - (\nabla e, \omega^2 \nabla \zeta)_\Omegah - (e, \omega^2 \zeta)_\Omegah,
\end{equation}
and thus we have
\begin{equation}
|J_2| \le \theta^2 d^2 \| \omega e_t \|_\Omegah^2
+ \frac{1}{2} \left( \| \omega\nabla e\|_\Omegah^2 + \| \omega e \|_\Omegah^2 \right)
+ C \left( \| \nabla \zeta \|_{D_d}^2 + d^{-2} \| \zeta \|_{D_d}^2 \right)
+ C d^{-2} \| e \|_{D_d}^2
\label{eq:J2}
\end{equation}
for arbitrary $\theta > 0$, since $\zeta_h = e + \zeta$.
To address $J_1$, we recall the asymptotic Galerkin orthogonality \eqref{eq:ago} and we have
\begin{align}
J_1 &= (e_t, \omega^2 \zeta_h - \chi)_\Omegah + a_\Omegah(e, \omega^2\zeta_h - \chi)
- (\tz_t + A \tz, \chi)_\OmghOmg - (\partial_\nh \tz, \chi)_{\partial\Omegah} \\
&=: \sum_{i=1}^{4} J_{1,i}
\end{align}
for arbitrary $\chi \in V_h$.
We choose $\chi = I_h(\omega^2 \zeta_h)$ so that $\supp \chi \subset D_{2d}$ if $d \ge h$.
We remark that the super-approximation type estimates
\begin{align}
\| \omega^2 \zeta_h - \chi \|_{D_{2d}}
&\le C h d^{-1} \| \zeta_h \|_{D_{2d}},
\label{eq:sc1}\\
\| \nabla (\omega^2 \zeta_h - \chi) \|_{D_{2d}}
&\le C \left( hd^{-2} \| \zeta_h \|_{D_{2d}} + hd^{-1} \| \nabla \zeta_h \|_{D_{2d}} \right)
\label{eq:sc2}
\end{align}
hold (see \cite[page~386]{ThoW00}).
Thus, $J_{1,1}$ and $J_{1,2}$ can be addressed as in \cite{SchTW98,ThoW00} and we have
\begin{equation}
|J_{1,1}| + |J_{1,2}|
\le C \left( \| \zeta \|_{1,D_{2d}}^2 + d^{-2} \| \zeta \|_{D_{2d}}^2 \right)
+ C \left( h^2 \| e_t \|_{D_{2d}}^2 + h d^{-1} \| e \|_{1,D_{2d}}^2 \right)
+ C d^{-2} \| e \|_{D_{2d}}^2.
\label{eq:J1112}
\end{equation}

In order to address $J_{1,3}$ and $J_{1,4}$, which are additional terms induced by the boundary layer, we state the following stability estimates with scaling:
\begin{align}
\| I_h(\omega^2 \zeta_h) \|_{D_{2d}} 
&\le C \| \zeta_h \|_{D_{2d}} ,
\label{eq:L2} \\
\| \nabla I_h(\omega^2 \zeta_h) \|_{D_{2d}} 
&\le C d^{-1} \| \zeta_h \|_{D_{2d}} + C \| \nabla \zeta_h \|_{D_{2d}},
\label{eq:nabla} \\
\| I_h( \omega^2 \zeta_h ) \|_{\dOmegah}
&\le C d^{1/2} \left( d^{-1} \| \zeta_h \|_{D_{2d}} + \| \nabla \zeta_h \|_{D_{2d}}  \right)
\label{eq:trace} 
\end{align}
The first two inequalities are derived from the super-approximation estimates \eqref{eq:sc1} and \eqref{eq:sc2}.
Indeed, \eqref{eq:sc2} yields
\begin{align}
\| \nabla I_h(\omega^2 \zeta_h) \|_{D_{2d}} 
&\le C \left( hd^{-2} \| \zeta_h \|_{D_{2d}} + hd^{-1} \| \nabla \zeta_h \|_{D_{2d}} \right)
 + \| \nabla (\omega^2 \zeta_h) \|_{D_{2d}} \\
&\le C d^{-1} \| \zeta_h \|_{D_{2d}} + C \| \nabla \zeta_h \|_{D_{2d}}
\end{align}
since $hd^{-1} \le 1$, which gives \eqref{eq:nabla}.
One can see that \eqref{eq:L2} is obtained more easily.
The third estimate \eqref{eq:trace} is derived from the trace inequality
\begin{equation}
\| I_h( \omega^2 \zeta_h ) \|_{\dOmegah}
\le C \| I_h(\omega^2 \zeta_h )\|_\Omegah^{1/2} \| I_h(\omega^2 \zeta_h )\|_{1,\Omegah}^{1/2}
\end{equation}
together with \eqref{eq:L2} and \eqref{eq:nabla}.
Here, the constant of the trace inequality may depend on $\Omegah$.
However, it can be bounded uniformly in $h$ since $\cT_h$ is quasi-uniform.

Now, we address $J_{1,3}$.
The inequality \eqref{eq:tube-4} yields
\begin{equation}
\| I_h(\omega^2 \zeta_h) \|_\OmghOmg
\le C \left( h \| I_h(\omega^2 \zeta_h) \|_\dOmegah + h^2 \| \nabla I_h(\omega^2 \zeta_h) \|_{D_{2d}} \right).
\end{equation}
Substituting \eqref{eq:nabla} and \eqref{eq:trace}, we have
\begin{equation}
\| I_h(\omega^2 \zeta_h) \|_\OmghOmg
\le C hd^{1/2} \left( d^{-1} \| \zeta_h \|_{D_{2d}} + \| \nabla \zeta_h \|_{D_{2d}}  \right),
\label{eq:2}
\end{equation}
which implies
\begin{equation}
|J_{1,3}| \le C h^2 d \| \tz_t + A \tz \|_{T(\varepsilon) \cap D_{2d}}^2 
+ C d^{-2} \| \zeta \|_{D_{2d}}^2 + C \| \nabla \zeta \|_{D_{2d}}^2 
+ C d^{-2} \| e \|_{D_{2d}}^2 + \theta_1^2 \| \nabla e \|_{D_{2d}}^2
\label{eq:J13}
\end{equation}
for arbitrary $\theta_1 > 0$ by the Young inequality.
Here, and hereafter, some constants may depend on $\theta_1$ and the dependency should be clarified.
However, it is dropped since we later choose $\theta_1$ independently of $h$, $d$, $z$, and $T$.
Moreover, we assume that $\theta_1$ is small and thus we will drop the coefficients of $\theta_1$,
since we can make $\theta_1$ smaller if necessary.

The estimate \eqref{eq:trace} also gives the bound for $J_{1,4}$.
Indeed, we have
\begin{align}
|J_{1,4}|
&\le \| \partial_\nh \tz \|_{\dOmegah \cap D_{2d}} 
\times C d^{1/2} \left( d^{-1} \| \zeta_h \|_{D_{2d}} + \| \nabla \zeta_h \|_{D_{2d}}  \right) \\
&\le C d \| \partial_\nh \tz \|_{\dOmegah \cap D_{2d}}^2 
+ C d^{-2} \| \zeta \|_{D_{2d}}^2 + C \| \nabla \zeta \|_{D_{2d}}^2 
+ C d^{-2} \| e \|_{D_{2d}}^2 + \theta_1^2 \| \nabla e \|_{D_{2d}}^2
\label{eq:J14}
\end{align}
for the same $\theta_1$ as above.

Therefore, from equations \eqref{eq:J2}, \eqref{eq:J1112}, \eqref{eq:J13}, and \eqref{eq:J14}, we can derive  
\begin{multline}
\frac{d}{dt} \| \omega e \|_\Omegah^2 + \| \omega \nabla e \|_\Omegah^2 + \| \omega e \|_\Omegah^2 \\
\le  \theta^2 d^2 \| \omega e_t \|_{D_{d}}^2 + \theta_1^2 \| \nabla e \|_{D_{2d}}^2
+ C d^{-2} \| e \|_{D_{2d}}^2  
+C d^2 \left( \bar{H}_{D_{2d}} + \bar{X}^{(1)}_{D_{2d}} + \bar{G}^{(1)}_{D_{2d}} \right),
\label{eq:energy-1}
\end{multline} 
where
\begin{align}
\bar{H}_{D_{2d}} &:= C_*^{-1} \left( \| e_t \|_{D_{2d}}^2 + d^{-2} \| e \|_{1,D_{2d}}^2 \right),
\label{eq:Hbar} \\
\bar{X}^{(1)}_{D_{2d}} &:= d^{-2} \| \zeta \|_{1,D_{2d}}^2 + d^{-4} \| \zeta \|_{D_{2d}}^2, \\
\bar{G}^{(1)}_{D_{2d}} &:= h^2 d^{-3} \| \tz_t + A \tz \|_{T(\varepsilon) \cap D_{2d}}^2 + d^{-3} \| \partial_\nh \tz \|_{\dOmegah \cap D_{2d}}^2.
\end{align}
Here we used the fact that $hd^{-1} \le C_*^{-1} \le 1$. 
Integrating \eqref{eq:energy-1} over $I_d$, multiplying it by $d^{-2}$, and taking the square roots, we have 
\begin{multline}
d^{-1} \trplnorm{e}{1,Q}
\le
\kappa_d I_{D_d}
+ \theta \trplnorm{\omega e_t}{Q_{h,T}}
+ \theta_1 d^{-1} \trplnorm{e}{1,Q_{2d}}
+ C d^{-2} \trplnorm{e}{Q_{2d}} \\
+ C \left( X_{Q_{2d}} 
+ H_{Q_{2d}} 
+ G_{Q_{2d}} \right). 
\label{eq:energy-2}
\end{multline}

\textit{Step 2.} We next consider the local $H^1$-$L^2$-estimate. 
Note that the strategy of this step is different from the literature due to the effect of the boundary layer.
From a basic calculation, we have 
\begin{equation}
\| \omega e_t \|_\Omegah^2 + \frac{1}{2} \frac{d}{dt} \left( \| \omega \nabla e \|_\Omegah^2 + \| \omega e \|_\Omegah^2 \right)
= K_1 + K_2,
\end{equation}
where
\begin{equation}
K_1 = (e_t, (\omega^2 \zeta_h)_t)_\Omegah + a_\Omegah(e,(\omega^2 \zeta_h)_t)
\end{equation}
and
\begin{multline}
K_2 =  -(e_t, \omega^2 \zeta_t)_\Omegah 
- (\nabla e, \omega^2 \nabla \zeta_t)_\Omegah - (e,\omega^2 \zeta_t)_\Omegah 
-(e, 2 \omega\omega_t \zeta)_\Omegah 
\\
- (e_t, 2 \omega\omega_t \zeta_h )_\Omegah 
- (\nabla e, \nabla \partial_t (\omega^2) \zeta_h)_\Omegah + (\nabla e, 2 \omega\omega_t \nabla e)_\Omegah.
\end{multline}
The second term $K_2$ can be addressed by the Young inequality and we have
\begin{equation}
K_2 
\le \frac{1}{2} \| \omega e_t \|_\Omegah^2 + C \left( d^2 \| \nabla \zeta_t \|_{D_d}^2 + \| \zeta_t \|_{D_d}^2 + d^{-4} \| \zeta \|_{D_d} \right)
+ C d^{-2} \| \nabla e \|_{D_d}^2 + C d^{-4} \| e \|_{D_d}^2.
\label{eq:K2}
\end{equation}
As in the case of $J_1$, the asymptotic Galerkin orthogonality \eqref{eq:ago} gives
\begin{align}
K_1 &= (e_t, (\omega^2 \zeta_h)_t - \chi)_\Omegah + a_\Omegah(e, (\omega^2 \zeta_h)_t - \chi)
- (\tz_t + A \tz, \chi)_\OmghOmg - (\partial_\nh \tz, \chi)_{\partial\Omegah} \\
&=: \sum_{i=1}^4 K_{1,i}
\end{align}
for arbitrary $\chi \in V_h$.
We choose $\chi = I_h [(\omega^2 \zeta_h)_t] = [I_h(\omega^2 \zeta_h)]_t$.
Then, we need new super-approximation estimates for this $\chi$ as follows:
\begin{align}
\| (\omega^2 \zeta_h)_t - I_h[(\omega^2 \zeta_h)_t] \|_{D_{2d}}
&\le C h d^{-3} \| \zeta_h \|_{D_{2d}} + hd^{-1} \| \zeta_{h,t} \|_{D_{2d}} ,
\label{eq:sc3}\\
\| \nabla ( (\omega^2 \zeta_h)_t - I_h[(\omega^2 \zeta_h)_t] ) \|_{D_{2d}}
&\le C d^{-3} \| \zeta_h \|_{D_{2d}} + d^{-1} \| \zeta_{h,t} \|_{D_{2d}} 
\label{eq:sc5}
\end{align}
for all $\zeta_h \in C^1(0,T;V_h)$.
Here we show \eqref{eq:sc3} only since the derivation of \eqref{eq:sc5} is similar.
For each element $K \subset D_{2d}$, it is clear that
\begin{equation}
\| (\omega^2 \zeta_h)_t - I_h[(\omega^2 \zeta_h)_t] \|_K
\le C h^{N/2} h^{k+1} \| \nabla^{k+1} [(\omega^2 \zeta_h)_t] \|_{L^\infty(K)}.
\end{equation}
Expanding the right hand side and using the inverse inequalities, we have
\begin{multline}
\| \nabla^{k+1} [(\omega^2 \zeta_h)_t] \|_{L^\infty(K)}
\le C d^{-k-3} \| \zeta_h \|_{L^\infty(K)} + C(d^{-k-2} + h^{-k+1}d^{-3}) \| \nabla \zeta_h \|_{L^\infty(K)} \\
+ C(d^{-k-1} + h^{-1}d^{-k} + h^{-k}d^{-1}) \| \zeta_{h,t} \|_{L^\infty(K)}
\end{multline}
since $\nabla^{k+1} \zeta_h \equiv 0$. 
Thus, together with the inverse inequality again, we can obtain \eqref{eq:sc3}.

Let us go back to the estimate of $K_1$.
Using \eqref{eq:sc3} and \eqref{eq:sc5}, we can address $K_{1,1}$ and $K_{1,2}$ as
\begin{multline}
K_{1,1} + K_{1,2}
\le 
\theta_2^2 \| e_t \|_{D_{2d}}^2 + C d^{-2} \| e \|_{1,D_{2d}}^2 \\
+C \left( \| \zeta_t \|_{D_{2d}}^2 + d^{-2} \| \nabla \zeta \|_{D_{2d}} + d^{-4} \| \zeta \|_{D_{2d}} \right)
+ C hd^{-1} \| e_t \|_{D_{2d}}^2
\label{eq:K1112}
\end{multline}
for arbitrary $\theta_2 > 0$.
Here, we treat $\theta_2$ in the same manner as $\theta_1$ mentioned above.
In contrast to $J_{1,3}$ and $J_{1,4}$, we postpone treating $K_{1,3}$ and $K_{1,4}$.
Summarizing \eqref{eq:K2} and \eqref{eq:K1112}, and kicking-back the term involving $\omega e_t$, we obtain
\begin{multline}
\| \omega e_t \|_\Omegah^2 + \frac{d}{dt} \left( \| \omega \nabla e \|_\Omegah^2 + \| \omega e \|_\Omegah^2 \right)
\\
\le 
\theta_2^2 \| e_t \|_{D_{2d}}^2 + C d^{-2} \| \nabla e \|_{D_d}^2 + C d^{-4} \| e \|_{D_d}^2 \\
+ C \left( \bar{H}_{D_{2d}} + \bar{X}^{(2)}_{D_{2d}} \right)
+ K_{1,3} + K_{1,4},
\end{multline}
where
\begin{equation}
\bar{X}^{(2)}_{D_{2d}} := d^2 \| \nabla \zeta_t \|_{D_{2d}}^2 + \| \zeta_t \|_{D_{2d}}^2 + d^{-2} \| \nabla \zeta \|_{D_{2d}}^2 + d^{-4} \| \zeta \|_{D_{2d}}^2
\end{equation}
and $\bar{H}_{D_{2d}}$ is defined by \eqref{eq:Hbar}.
Integrating both sides over $I_d$, we have
\begin{multline}
\trplnorm{\omega e_t}{Q_d}^2 
+ \lambda_d \left( \| \omega_1 \nabla e(T) \|_{\Omegah}^2 + \| \omega_1 e(T) \|_{\Omegah}^2 \right) \\
\le \kappa_d I_{D_{2d}}^2
+ \theta_2^2 \trplnorm{e_t}{Q_{2d}}^2 
+ C d^{-2} \trplnorm{e}{1,Q_{2d}}^2 + C d^{-4} \trplnorm{e}{Q_{2d}}^2 \\
+ C \left( H_{Q_{2d}}^2 + X_{Q_{2d}}^2 \right)
+ \cK_3 + \cK_4,
\label{eq:3}
\end{multline}
where $\cK_{i} = \int_{I_d} K_{1,i} dt$ ($i=3,4$).

We address $\cK_3$ and $\cK_4$ by integration by parts.
Since $\tz(0)|_{T(\varepsilon)} \equiv 0$, we have
\begin{align}
\cK_3 
&= - \int_{I_d} (\tz_t + A\tz, [I_h(\omega^2 \zeta_h)]_t)_\OmghOmg dt \\
&= \int_{I_d} (\tz_{tt} + A\tz_t, I_h(\omega^2 \zeta_h))_\OmghOmg dt
- \lambda_d (\tz_t(T) + A\tz(T), I_h(\omega_1^2 \zeta_h(T)) )_\OmghOmg,
\end{align}
where $\lambda_d$ is defined by \eqref{eq:kappa-lambda}.
Recalling \eqref{eq:2}, we obtain
\begin{align}
\cK_3 &\le 
\begin{multlined}[t] 
\trplnorm{\tz_{tt} + A\tz_t}{L_T(\varepsilon) \cap Q_{2d}} 
\times C hd^{1/2} \left( d^{-1} \trplnorm{\zeta_h}{Q_{2d}} + \trplnorm{\nabla \zeta_h}{Q_{2d}} \right) \\
 + \lambda_d \| \tz_t(T) + A\tz(T) \|_{T(\varepsilon) \cap D_{2d}}
 \times C hd^{1/2} \left( d^{-1} \| \zeta_h(T) \|_{D_{2d}} + \| \nabla \zeta_h(T) \|_{D_{2d}} \right)
\end{multlined}
\\[1ex]
&\le 
\begin{multlined}[t] 
C h^2 d^3 \trplnorm{\tz_{tt} + A\tz_t}{L_T(\varepsilon) \cap Q_{2d}}^2 
+ C d^{-4} \trplnorm{\zeta_h}{Q_{2d}}^2
+ C d^{-2} \trplnorm{\zeta_h}{1,Q_{2d}}^2 \\
+ \lambda_d \left( 
  C h^2 d \| \tz_t(T) + A\tz(T) \|_{T(\varepsilon) \cap D_{2d}}^2
+ \theta_2^2 d^{-2} \| \zeta_h(T) \|_{D_{2d}}^2 
+ \theta_2^2 \| \nabla \zeta_h(T) \|_{D_{2d}}^2  \right)
\end{multlined}
\end{align}
for the same $\theta_2>0$ as above.
Here we used the fact that
\begin{equation}
\| I_h(\omega_1^2 \zeta_h(T)) \|_\OmghOmg
\le C hd^{1/2} \left( d^{-1} \| \zeta_h(T) \|_{D_{2d}} + \| \nabla \zeta_h(T) \|_{D_{2d}} \right),
\end{equation}
which is derived by letting $t=T$ in \eqref{eq:2}.
Moreover, since $[T-d^2, T] \subset I_d$ provided $\lambda_d = 1$, we obtain the trace inequality of the form
\begin{align}
\| \psi(T) \|_{D_{2d}} 
&\le C d^{-1} \trplnorm{\psi}{D_{2d} \times [T-d^2,T]} + C d \trplnorm{\psi_t}{D_{2d} \times [T-d^2,T]} \\
&\le C d^{-1} \trplnorm{\psi}{Q_{2d}} + C d \trplnorm{\psi_t}{Q_{2d}}
\label{eq:trace-time}
\end{align}
for any $\psi \in H^1(I_{2d}; L^2(D_{2d}))$.
Therefore, together with $\| \zeta_h \| \le \| \zeta \| + \| e \|$, 
we can merge several terms involving $\lambda_d$ and we obtain
\begin{equation}
\cK_3 \le
\theta_2^2 \lambda_d \| \nabla e(T) \|_{D_{2d}}^2 
+ \theta_2^2 \trplnorm{e_t}{Q_{2d}}^2
+ C d^{-2} \trplnorm{e}{1,Q_{2d}}^2 + Cd^{-4} \trplnorm{e}{Q_{2d}}^2 
+ C \left( X_{Q_{2d}}^2 + G_{Q_{2d}}^2 \right)
\label{eq:K3}
\end{equation}

We repeat this calculation. By integration by parts, we have
\begin{align}
\cK_4 
&= - \int_{I_d} (\partial_\nh \tz, [I_h(\omega^2 \zeta_h)]_t)_\dOmegah dt \\
&= \int_{I_d} (\partial_\nh \tz_t, I_h(\omega^2 \zeta_h) )_\dOmegah dt 
- \lambda_d (\partial_\nh \tz(T), I_h(\omega_1^2 \zeta_h(T)) )_\dOmegah
\end{align}
Using \eqref{eq:trace} (for general time $t$ and the specified time $t=T$), we have
\begin{align}
\cK_4 &
\le
\begin{multlined}[t]
\trplnorm{\partial_\nh \tz_t}{\Sigma_{h,T} \cap Q_{2d}} 
\times C d^{1/2} \left( d^{-1} \trplnorm{\zeta_h}{Q_{2d}} + \trplnorm{\nabla \zeta_h}{Q_{2d}} \right) \\
+ \lambda_d \| \partial_\nh \tz (T) \|_{\dOmegah \cap D_{2d}}
\times C d^{1/2} \left( d^{-1} \| \zeta_h(T) \|_{D_{2d}} + \| \nabla \zeta_h(T) \|_{D_{2d}} \right) 
\end{multlined} 
\\[1ex]
&\le 
\begin{multlined}[t]
C d^3 \trplnorm{\partial_\nh \tz_t}{\Sigma_{h,T} \cap Q_{2d}}^2 
+ C d^{-4} \trplnorm{\zeta_h}{Q_{2d}}^2 
+ C d^{-2} \trplnorm{\zeta_h}{1,Q_{2d}}^2 \\
+ \lambda_d \left( C d \| \partial_\nh \tz(T) \|_{\dOmegah \cap D_{2d}}^2 
+ \theta_2^2 d^{-2} \| \zeta_h(T) \|_{D_{2d}}^2
+ \theta_2^2 \| \nabla \zeta_h(T) \|_{D_{2d}}^2 \right).
\end{multlined}
\end{align}
Again using the trace inequality in time \eqref{eq:trace-time} and merging several terms, we obtain
\begin{equation}
\cK_4 \le
\theta_2^2 \lambda_d \| \nabla e(T) \|_{D_{2d}}^2 
+ \theta_2^2 \trplnorm{e_t}{Q_{2d}}^2
+ C d^{-2} \trplnorm{e}{1,Q_{2d}}^2 + Cd^{-4} \trplnorm{e}{Q_{2d}}^2 
+ C \left( X_{Q_{2d}}^2 + G_{Q_{2d}}^2 \right)
\label{eq:K4}
\end{equation}
Substituting \eqref{eq:K3} and \eqref{eq:K4} into \eqref{eq:3} and taking square roots, we have
\begin{multline}
\trplnorm{\omega e_t}{Q_d}
+ \lambda_d \| e(T) \|_{1,D} \\
\le \kappa_d I_{D_{2d}}
+ \theta_2 \lambda_d \| e(T) \|_{1,D_{2d}}
+ \theta_2 \trplnorm{e_t}{Q_{2d}}
+ C_0 d^{-1} \trplnorm{e}{1,Q_{2d}} + C d^{-2} \trplnorm{e}{Q_{2d}} \\
+ C \left( H_{Q_{2d}} + X_{Q_{2d}} + G_{Q_{2d}} \right),
\label{eq:energy-3}
\end{multline}
where the constant $C_0$ is independent of $h$, $d$, $D$, and $I$.

\textit{Step 3.} Now we complete the local energy error estimate.
Multiplying \eqref{eq:energy-3} by $2\theta$ and adding it to \eqref{eq:energy-2}, 
we can kick-back the term $\theta \trplnorm{\omega e_t}{Q_{h,T}}$ and obtain 
\begin{multline}
\theta \trplnorm{e_t}{Q} + d^{-1} \trplnorm{e}{1,Q} + 2\theta \lambda_d \| e(T) \|_{1,D} \\
\le
  2 \theta \theta_2 \lambda_d \| e(T) \|_{1,D_{2d}}
+ 2 \theta \theta_2 \trplnorm{e_t}{Q_{2d}}
+ \left( \theta_1 + C_0 \theta \right) d^{-1} \trplnorm{e}{1,Q_{2d}}
+ C d^{-2} \trplnorm{e}{Q_{2d}} \\
+ C \left( \kappa_d I_{D_d} + X_{Q_{2d}} + H_{Q_{2d}} + G_{Q_{2d}} \right).
%
\end{multline}
Since $\theta_1$ and $\theta_2$ are arbitrary positive numbers, 
we set $\theta_1 = C_0 \theta$ and $\theta_2 = C_0 \theta/2$.
Then, we obtain
\begin{multline}
\theta \trplnorm{e_t}{Q} + d^{-1} \trplnorm{e}{1,Q} + \theta \lambda_d \| e(T) \|_{1,D} \\
\le
  2C_0 \theta \left(
  \theta \trplnorm{e_t}{Q_{2d}} + d^{-1} \trplnorm{e}{1,Q_{2d}} + \theta \lambda_d \| e(T) \|_{1,D_{2d}}
  \right)
+ C d^{-2} \trplnorm{e}{Q_{2d}} \\
+ C \left( \kappa_d I_{D_d} + X_{Q_{2d}} + H_{Q_{2d}} + G_{Q_{2d}} \right).
\label{eq:energy-4}
\end{multline}
Finally, 
replacing $2d$ by $d$ and $2C_0$ by $C_0$, respectively, 
we can establish the desired estimate~\eqref{eq:local} and thus we complete the proof of \cref{lem:local}.
\end{proof}


\section{Duality argument}
\label{sec:dual}

In this section, we show \cref{lem:duality}.

\begin{proof}[Proof of \cref{lem:duality}.]
In this proof, we denote the space-time inner products by $[\cdot, \cdot]$.
For example,
\begin{equation}
[u,v]_{Q_{h,T}} = \iint_{Q_{h,T}} u(x,t) v(x,t)\, dx dt,
\quad
a_{Q_{h,T}}[u,v] = \iint_{Q_{h,T}} (\nabla_x u \cdot \nabla_x v + uv) \, dxdt.
\end{equation}
We recall that
\begin{equation}
\trplnorm{F}{Q_{h,j}} = \sup \{ [\phi, F]_{Q_{h,T}} \mid \phi \in C_0^\infty(\bR^{N+1}),\ \supp \phi \subset Q_{h,j}, \ \trplnorm{\phi}{Q_{h,j}} = 1 \} .
\end{equation}
We fix such $\phi \in C_0^\infty(Q_{h,j})$ and consider the dual parabolic problem
\begin{equation}
\begin{cases}
-\partial_t w + A w = \phi, & \text{in } Q_T, \\
\partial_n w = 0, & \text{on } \partial\Omega \times (0,T), \\
w(T) = 0, & \text{in } \Omega.
\end{cases}
\end{equation}
Then, in analogy with \cref{lem:reduce}, we state
\begin{equation}
[\phi, F]_{Q_{h,T}} = (\tw(0), F(0))_\Omegah + \sum_{l=0}^{6} E'_l,
\label{eq:reduce-dual}
\end{equation}
where
\begin{align}
E'_0 &= [\tw - w_h, F_t ]_{Q_{h,T}}
+ a_{Q_{h,T}}[\tw - w_h, F],
& & \\
E'_1 &= [\tw - w_h, \tGamma_t + A\tGamma ]_{Q_{h,T} \setminus Q_T},
&
E'_2 &= [\tw - w_h, \partial_{n_h} \tGamma ]_{\Sigma_{h,T}},
\\
E'_3 &= [\phi + \tw_t - A\tw, F]_{Q_{h,T} \setminus Q_T},
&
E'_4 &= [ - \partial_{n_h} \tw, F]_{\Sigma_{h,T}},
\\
E'_5 &= [ \tw_t, \tGamma]_{Q_{h,T} \setminus Q_T} - [w_t, \Gamma]_{Q_T \setminus Q_{h,T}},
& & \\
E'_6 &= a_{Q_T \setminus Q_{h,T}}[w,\Gamma] - a_{Q_{h,T} \setminus Q_T}[\tw, \tGamma].
& &
\end{align}
for arbitrary $w_h \in V_h$.
We present an outline of its proof.
Noting that $\phi|_{Q_T \setminus Q_{h,T}} \equiv 0$, we have
\begin{align}
[\phi, F]_{Q_{h,T}}
&= [\phi, \tilde{F}]_{Q_T} + [\phi, F]_{Q_{h,T} \setminus Q_T} \\
&= [-w_t, \tilde{F}]_{Q_T} + a_{Q_T}[w, \tilde{F}] + [\phi, F]_{Q_{h,T} \setminus Q_T}
\end{align}
from identity \eqref{eq:gap}.
Again applying \eqref{eq:gap}, integrating by parts both in time and space, and recalling the asymptotic Galerkin orthogonality \eqref{eq:ago}, we have
\begin{equation}
[\phi, F]_{Q_{h,T}} = (\tw(0), F(0))_\Omegah + E'_0 + E'_3 + E'_4
- [w_h, \tGamma_t + A\tGamma]_{Q_{h,T} \setminus Q_T} - [w_h, \partial_\nh \tGamma]_{\Sigma_{h,T}}
\end{equation}
for arbitrary $w_h \in V_h$.
Adding the null terms
\begin{equation}
[\tw, \tGamma_t + A\tGamma]_{Q_{h,T} \setminus Q_T}
- [\tw, \tGamma_t + A\tGamma]_{Q_{h,T} \setminus Q_T}
+ [w, \Gamma_t + A\Gamma]_{Q_T \setminus Q_{h,T}}
(=0)
\end{equation}
to the right hand side, we can obtain \eqref{eq:reduce-dual}.

By estimating each terms in \eqref{eq:reduce-dual}, we show \eqref{eq:dual}.
The treatment of $(\tw(0), F(0))_\Omegah$ is the same as in \cite[Lemma~4.2]{SchTW98} and we have
\begin{equation}
|(\tw(0), F(0))_\Omegah| \le C h^2 d_j^{-\frac{N}{2}-1}.
\label{eq:D-1}
\end{equation}
For the estimates of $E'_l$, we choose $w_h = \tilde{I}_h \tw$, where $\tilde{I}_h$ is the quasi-interpolation operator introduced in Section~\ref{sec:preliminaries}.
Then, $E'_0$ can be addressed as in \cite{SchTW98} and we have
\begin{align}
|E'_0|
&\le \sum_{i,*} \left( \trplnorm{\tw - \tilde{I}_h \tw}{Q_{h,i}} \trplnorm{F_t}{Q_{h,i}} + \trplnorm{\tw - \tilde{I}_h \tw}{1,Q_{h,i}} \trplnorm{F}{1,Q_{h,i}} \right) \\
&\le C \sum_{i,*} \left( h^2 \trplnorm{F_t}{Q_{h,i}} + h \trplnorm{F}{1,Q_{h,i}} \right) \min \left\{ \left( \frac{d_j}{d_i} \right)^{\frac{N}{2}+1}, \left( \frac{d_i}{d_j} \right)^{\frac{N}{2}+1} \right\},
\label{eq:D0}
\end{align}
since $\trplnorm{\tw}{2,Q_{h,i}} \le C \min \{ (d_j d_i^{-1})^{N/2 + 1}, (d_i d_j^{-1})^{N/2 + 1} \}$
owing to \eqref{eq:ext-local} and \eqref{eq:gaussian}.
In order to address other terms, we set $Q''_{h,j} := Q'_{h,j-1} \cup Q'_{h,j} \cup Q'_{h,j+1}$ and $Q'''_{h,j} := Q''_{h,j-1} \cup Q''_{h,j} \cup Q''_{h,j+1}$.
We decompose $E'_1$ as
\begin{align}
E'_1 = [\tw - \tilde{I}_h \tw, \tGamma_t + A \tGamma]_{Q''_{h,j} \setminus Q_T} + [\tw - \tilde{I}_h \tw, \tGamma_t + A \tGamma]_{Q_{h,T} \setminus (Q_T \cup Q''_{h,j})}
=: E'_{1,1} + E'_{1,2}.
\end{align}
Since $\trplnorm{\tw}{2,Q_{h,T}} \le C \trplnorm{\phi}{Q_{h,T}} = C$ by the standard energy estimate, we have, together with \eqref{eq:gap-1},
\begin{equation}
|E'_{1,1}| \le C h^2 \trplnorm{\tGamma_t + A \tGamma}{L_T(\varepsilon) \cap Q''_{h,j}} \le C h^3 d_j^{-\frac{N}{2}-\frac{3}{2}}.
\end{equation}
From \cref{lem:approx}, \eqref{eq:E1-1}, and \eqref{eq:ext-local}, we have
\begin{equation}
|E'_{1,2}| \le C h^2 \| w \|_{W^{2,\infty}(Q_T \setminus Q'_{h,j})}.
\end{equation}
Since we can write
\begin{equation}
w(x,t)= \int_t^T \int_\Omega G(x,y; s-t) \phi(y,s) \, dy ds,
\label{eq:w}
\end{equation}
the Gaussian estimate \eqref{eq:gaussian} and the assumption $\supp \phi \subset Q_{h,j}$ yield
\begin{equation}
\| w \|_{W^{m,\infty}(Q_T \setminus Q'_{h,j})}
\le C |Q_{h,j}|^{1/2} d_j^{-N-m} \le d_j^{-\frac{N}{2}+1-m}.
\label{eq:w-Qj}
\end{equation}
Therefore, we have
\begin{equation}
|E'_1| \le C h^3 d_j^{-\frac{N}{2}-\frac{3}{2}} + C h^2 d_j^{-\frac{N}{2}-1}
\le C h^2 d_j^{-\frac{N}{2}-1}
\label{eq:D1}
\end{equation}
since $h d_j^{-1} \le C_*^{-1} \le 1$.
The estimate of $E'_2$ is similar.
Indeed, we divide $E'_2$ into two parts $E'_2 = E'_{2,1} + E'_{2,2}$, where
\begin{equation}
E'_{2,1} = [\tw - w_h, \partial_{n_h} \tGamma ]_{\Sigma_{h,T} \cap Q''_{h,j}}, \qquad
E'_{2,2} = [\tw - w_h, \partial_{n_h} \tGamma ]_{\Sigma_{h,T} \setminus Q''_{h,j}}.
\end{equation}
Recalling the scaled trace inequality
\begin{equation}
\trplnorm{\psi}{\Sigma_{h,T} \cap Q''_{h,j}}
\le C d_j^{1/2} ( d_j^{-1} \trplnorm{\psi}{Q'''_{h,j}} + \trplnorm{\psi}{1,Q'''_{h,j}} ),
\label{eq:trace-scaled}
\end{equation}
we have
\begin{equation}
\trplnorm{\tw - \tilde{I}\tw}{\Sigma_{h,T} \cap Q''_{h,j}}
\le C d_j^{1/2} h(hd_j^{-1} + 1) \trplnorm{w}{2,Q_T}
\le C hd_j^{1/2}.
\end{equation}
Moreover, by the same calculation as in \eqref{eq:tmp1}, we have
\begin{equation}
\trplnorm{\partial_\nh \tGamma}{\Sigma_{h,T} \cap Q''_{h,j}}
\le C h \left( \trplnorm{\nabla \tGamma}{\Sigma_{h,T} \cap Q'''_{h,j}}
+ \trplnorm{\tGamma}{2,L_T(\varepsilon) \cap Q'''_{h,j}}  \right),
\end{equation}
and the boundary-skin estimates \eqref{eq:gap-1} and \eqref{eq:gap-2} give
\begin{equation}
\trplnorm{\partial_\nh \tGamma}{\Sigma_{h,T} \cap Q'''_{h,j}}
\le C h d_j^{-\frac{N}{2} - \frac{1}{2}},
\end{equation}
which implies
\begin{equation}
|E'_{2,1}| \le C h^2 d_j^{-\frac{N}{2}}.
\end{equation}
Further, since $\| \partial_\nh\tGamma \|_{L^1(\Sigma_{h,T})} \le C h |\log h|$ (\cref{lem:gamma}), we have
\begin{equation}
|E'_{2,2}| \le C h^2 \| w \|_{W^{2,\infty}(Q_T \setminus Q'_{h,j})} \| \partial_\nh\tGamma \|_{L^1(\Sigma_{h,T})}
\le C h^3|\log h| d_j^{-\frac{N}{2}-1}.
\end{equation}
Hence we have
\begin{equation}
|E'_2| \le C h^2 d_j^{-\frac{N}{2} - 1}.
\label{eq:D2}
\end{equation}

We divide $E'_3$ into $E'_3 = E'_{3,1} + E'_{3,2}$, where
\begin{equation}
E'_{3,1} = [-\tw_t + A\tw - \phi, F]_{Q''_{h,j} \setminus Q_T},
\quad
E'_{3,2} = [-\tw_t + A\tw, F]_{Q_{h,T} \setminus (Q_T \cup Q''_{h,j})}.
\end{equation}
From the energy estimates, $\trplnorm{ {-}\tw_t + A\tw - \phi}{Q_{h,T}} \le C$.
Moreover, from \eqref{eq:tube-4-loc} and the scaled trace inequality \eqref{eq:trace-scaled}, we have
\begin{equation}
\trplnorm{F}{Q''_{h,i} \setminus Q_T}
\le C (hd_j^{-1} \trplnorm{F}{Q'''_{h,j}} + h \trplnorm{\nabla F}{Q'''_{h,j}}).
\end{equation}
Thus we have
\begin{equation}
|E'_{3,1}| \le C (hd_j^{-1} \trplnorm{F}{Q'''_{h,j}} + h \trplnorm{\nabla F}{Q'''_{h,j}}).
\end{equation}
The expression \eqref{eq:w} and the Gaussian estimate \eqref{eq:gaussian} yield
\begin{equation}
\| {-}\tw_t + A\tw \|_{L^\infty(Q_{h,T} \setminus (Q_T \cup Q''_{h,j}))} \le C d_j^{-\frac{N}{2}-1}
\end{equation}
and \eqref{eq:tube-4} implies
\begin{equation}
\| F \|_{L^1(Q_{h,T} \setminus (Q_T \cup Q''_{h,j}))} \le C h^2 \| F \|_{L^1(0,T; W^{1,1}(\Omega_h))}.
\end{equation}
Hence we have
\begin{equation}
|E'_{3,2}| \le C h^2 d_j^{-\frac{N}{2}-1} \| F \|_{L^1(0,T; W^{1,1}(\Omega_h))},
\end{equation}
which yields
\begin{equation}
|E'_3| \le C (hd_j^{-1} \trplnorm{F}{Q'''_{h,j}} + h \trplnorm{\nabla F}{Q'''_{h,j}})
+ C h^2 d_j^{-\frac{N}{2}-1} \| F \|_{L^1(0,T; W^{1,1}(\Omega_h))}.
\label{eq:D3}
\end{equation}
Similarly, we can observe
\begin{equation}
|E'_4| \le C (hd_j^{-1} \trplnorm{F}{Q'''_{h,j}} + h \trplnorm{\nabla F}{Q'''_{h,j}})
+ C h d_j^{-\frac{N}{2}} \| F \|_{L^1(0,T; W^{1,1}(\Omega_h))}
\label{eq:D4}
\end{equation}
owing to \eqref{eq:tube-2-loc} and the trace inequality with scaling \eqref{eq:trace-scaled}.
Here, we perform a calculation similar to \eqref{eq:tmp1} to address $\partial_{n_h} \tw$.

The treatment of $E'_5$ and $E'_6$ is the same as above.
Indeed, we have
\begin{equation}
|E'_5|
\le |[\tw_t, \tGamma]_{L_T(\varepsilon) \cap Q''_{h,j}}|
+ |[\tw_t, \tGamma]_{L_T(\varepsilon) \setminus Q''_{h,j}}| =: E'_{5,1} + E'_{5,2},
\end{equation}
with the estimates
\begin{equation}
E'_{5,1} \le C \trplnorm{w_t}{Q_T} \trplnorm{\tGamma}{L_T(\varepsilon) \cap Q''_{h,j}} \le C h d_j^{-\frac{N}{2}+\frac{1}{2}}
\end{equation}
from the boundary-skin estimate \eqref{eq:gap-1} and the energy estimate, and
\begin{equation}
E'_{5,2} \le C \| w_t \|_{L^\infty(Q_T \setminus Q'_{h,j})} \| \tGamma \|_{L^1(L_T(\varepsilon))}
\le C h^2 d_j^{-\frac{N}{2}-1}
\end{equation}
from \cref{lem:gamma} and the expression \eqref{eq:w}.
Thus we have
\begin{equation}
|E'_5| \le C h d_j^{-\frac{N}{2}+\frac{1}{2}}.
\label{eq:D5}
\end{equation}
Furthermore, we can write $|E'_6| \le E'_{6,1} + E'_{6,2}$, where
\begin{equation}
E'_{6,1} = \trplnorm{\tw}{1,L_T(\varepsilon) \cap Q''_{h,j}}
\trplnorm{\tGamma}{1,L_T(\varepsilon)\cap Q''_{h,j}}
\end{equation}
and
\begin{equation}
E'_{6,2} = \| \tw \|_{W^{1,\infty}(L_T(\varepsilon) \setminus Q''_{h,j})}
\| \tGamma \|_{W^{1,1}(L_T(\varepsilon) \setminus Q''_{h,j})}.
\end{equation}
From \eqref{eq:tube-3} and gap estimate \eqref{eq:gap-1}, we have
\begin{equation}
E'_{6,1} \le C h \trplnorm{w}{2,Q_T} \trplnorm{\tGamma}{1,L_T(\varepsilon) \cap Q''_{h,j}}
\le C h^2 d_j^{-\frac{N}{2} - \frac{1}{2}}.
\end{equation}
Also, \eqref{eq:w-Qj} and \eqref{eq:gap-1} yield
\begin{equation}
E'_{6,2} \le C \| w \|_{W^{1,\infty}(Q_T \setminus Q'_{h,j})} \sum_{i,*} \| \tGamma \|_{W^{1,1}(L_T(\varepsilon) \cap Q_{h,i})}
\le C h^2 |\log h| d_j^{-\frac{N}{2}}.
\end{equation}
Thus we have
\begin{equation}
|E'_6| \le C h d_j^{-\frac{N}{2} + \frac{1}{2}}.
\label{eq:D6}
\end{equation}
Summarizing \eqref{eq:D-1}, \eqref{eq:D0}, \eqref{eq:D1}, \eqref{eq:D2}, \eqref{eq:D3}, \eqref{eq:D4}, \eqref{eq:D5}, and \eqref{eq:D6}, we can obtain \eqref{eq:dual},
since we can replace $Q'''_{h,j}$ by $Q'_{h,j}$ in \eqref{eq:D3} and \eqref{eq:D4} by changing the width of extension of domains.
Hence we can complete the proof of \cref{lem:duality}.
\end{proof}

\section{Proofs of \cref{thm:semigroup,thm:dmr}}
\label{sec:cor}

At this stage, we can show \cref{thm:semigroup} in the same way as \cite[Proposition~3.2]{SchTW98}.
Indeed, it suffices to show $\| tF_{tt} \|_{L^1(Q_{h,T})} \le C$ for $T \le 1$, which can be obtained from \cref{lem:local} and an argument similar to the previous section.
Hence we omit the proof and we address \cref{thm:dmr} here.

\begin{proof}[Proof of \cref{thm:dmr}]
As mentioned in the last part of Section~\ref{sec:proof}, we may assume $T \le 1$. 
Moreover, it suffices to show \eqref{eq:dmr} for the case $p=q$ by the general theory of maximal regularity (cf.~\cite[Theorem~4.2]{Dor93}).

Let us recall that $u_h \in C^0([0,T]; V_h)$ is the solution of
\begin{equation}
\begin{cases}
(u_{h,t}(t), v_h)_{\Omega_h} + a_\Omegah(u_h(t), v_h) = (f_h(t), v_h)_{\Omega_h}, & \forall v_h \in V_h, \\
u_h(0) = 0,
\end{cases}
\end{equation}
for given $f_h \in L^p(0,T;V_h)$.
Thus we have a representation
\begin{equation}
u_h(t) = \int_0^t A_h e^{-(t-s)A_h} f_h(s) ds,
\end{equation}
which implies
\begin{equation}
(-A_h u_h)(x,t) = \int_0^t \int_\Omegah \partial_t \Gamma_{x,h}(y,t-s) f_h(y,s) dyds
=: (\partial_t \Gamma_{x,h} * f_h)(x,t)
,
\quad (x,t) \in Q_{h,T},
\end{equation}
where $\Gamma_{x,h}$ is the discretized regularized Green's function defined by \eqref{eq:disc-gamma} for $x_0 = x \in \Omega_h$.
Therefore, maximal regularity is equivalent to the $L^p(Q_{h,T})$-boundedness of the convolution operator with respect to $\partial_t \Gamma_{x,h}$.
Moreover, \cref{lem:green} yields
\begin{equation}
\| \partial_t \Gamma_{x,h} * f_h \|_{L^p(Q_{h,T})}
\le C \| f_h \|_{L^p(Q_{h,T})} + \| \partial_t \tGamma_x * f_h \|_{L^p(Q_{h,T})},
\end{equation}
where $\Gamma_x$ is regularized Green's function defined by \eqref{eq:gamma} with respect to $x_0 = x \in \Omegah$, $\tGamma_{x,h}$ is its appropriate extension to $\Omegah$, and
\begin{equation}
(\partial_t \tGamma_x * f_h)(x,t) :=
\int_0^t \int_\Omegah \partial_t \tGamma_x(y,t-s) f_h(y,s) dyds
,
\quad (x,t) \in Q_{h,T}.
\end{equation}
Thus, what remains to show is
\begin{equation}
\| \partial_t \tGamma_x * f_h \|_{L^p(Q_{h,T})}
\le C \| f_h \|_{L^p(Q_{h,T})},
\quad \forall f_h \in L^p(Q_{h,T})
\label{eq:dmr-1}
\end{equation}
uniformly with respect to $h$.

Let
\begin{equation}
(\partial_t \Gamma_x * f)(x,t) :=
\int_0^t \int_\Omega \partial_t \Gamma_x(y,t-s) f(y,s) dyds
,
\quad (x,t) \in Q_T
\end{equation}
for $f \in L^p(Q_T)$.
Then, from the argument in \cite[pp.~685--686]{Gei06}, we have
\begin{equation}
\| \partial_t \Gamma_x * f \|_{L^p(Q_T)}
\le C \| f \|_{L^p(Q_T)},
\quad \forall f \in L^p(Q_T)
\end{equation}
uniformly with respect to $h$ for $p \in (1,\infty)$.
Now, we show \eqref{eq:dmr-1}.
For $f_h \in L^p(0,T;V_h)$, let $\tf_h \in L^p(Q_T)$ be the zero-extension of $f_h$.
Then,
\begin{equation}
(\partial_t \tGamma_x * f_h)(x,t) =
(\partial_t \Gamma_x * \tf_h)(x,t)
+
\Phi(x,t)
\end{equation}
for $(x,t) \in Q_{h,T}$, where
\begin{equation}
\Phi(x,t) = \int_0^t \int_\OmghOmg \partial_t \tGamma_x(y,t-s) f_h(y,s) dyds.
\end{equation}
Thus, we have
\begin{align}
\| \partial_t \tGamma_x * f_h \|_{L^p(Q_{h,T})}
&\le \| \partial_t \Gamma_x * \tf_h \|_{L^p(Q_T)}
+ \| \partial_t \Gamma_x * \tf_h \|_{L^p(Q_{h,T} \setminus Q_T)}
+ \| \Phi \|_{L^p(Q_{h,T})} \\
&\le C \| f_h \|_{L^p(Q_{h,T})}
+ \| \partial_t \Gamma_x * \tf_h \|_{L^p(Q_{h,T} \setminus Q_T)}
+ \| \Phi \|_{L^p(Q_{h,T})}.
\label{eq:dmr-2}
\end{align}
As in the proof of the Young inequality for convolution operators, one can see
\begin{multline}
\| \partial_t \Gamma_x * \tf_h \|_{L^p(Q_{h,T} \setminus Q_T)}
\le
\max_{x \in \OmghOmg} \left( \iint_{Q_T} | \partial_t \Gamma_x (y,s) | dyds \right)^{1/p'} \\
\times
\max_{y \in \Omega} \left( \iint_{Q_{h,T} \setminus Q_T} | \partial_t \Gamma_x (y,t) | dxdt \right)^{1/p}
\| f_h \|_{L^p(Q_{h,T})}
\end{multline}
and
\begin{multline}
\| \Phi \|_{L^p(Q_{h,T})}
\le
\max_{x \in \Omega_h} \left( \iint_{Q_{h,T} \setminus Q_T} | \partial_t \tGamma_x (y,s) | dyds \right)^{1/p'} \\
\times
\max_{y \in \OmghOmg} \left( \iint_{Q_{h,T}} | \partial_t \tGamma_x (y,t) | dxdt \right)^{1/p}
\| f_h \|_{L^p(Q_{h,T})},
\end{multline}
where $p'$ fulfills $1/p + 1/p' = 1$.
Here, we should discuss the measurability and integrability of $\partial_t \Gamma_x (y,t)$ with respect to $(x,t) \in Q_{h,T}$.
Fix $K \in \cT_h$ arbitrarily.
Then, $x \mapsto \bdelta_x(y)$ is Lipschitz continuous with Lipschitz constant which may depend on $h$ and $y$ by its construction \cite[Appendix]{SchSW96}.
Thus, by the maximum principle, we have
\begin{equation}
\| \partial_t (\Gamma_{x_1}(\cdot, t) - \Gamma_{x_2}(\cdot,t)) \|_{L^\infty(\Omega)}
\le \| (-\Delta +1) (\bdelta_{x_1} - \bdelta_{x_2}) \|_{L^\infty(\Omega)}
\le C_h |x_1-x_2|
\end{equation}
for arbitrary $x_1, x_2 \in K$ and $t>0$.
Further, $\partial_t \Gamma_x(y,t)$ is sufficiently smooth with respect to $t>0$.
Therefore, the function $(x,t) \mapsto \partial_t \Gamma_x(y,t)$ is piecewise continuous for each $y$ and $h$, and thus measurable and integrable.

We here address $\iint_{Q_{h,T} \setminus Q_T} | \partial_t \Gamma_x (y,t) | dxdt$ only.
As in \eqref{eq:dyadic2}, we define $Q_j(y)$ and $Q_*(y)$ as the parabolic dyadic decomposition centered at $(y,0)$, i.e.,
\begin{equation}
Q_{h,j}(y) := \{ (x,t) \in Q_{h,T} \mid d_j \le \rho_y(x,t) \le 2d_j \},
\quad
Q_{h,*}(y) := \{ (x,t) \in Q_{h,T} \mid \rho_y(x,t) \le d_{J_*} \},
\end{equation}
where $\rho_y(x,t) = \max\{ |x-y|, \sqrt{t} \}$.
Then, as discussed in the proof of \cref{lem:gap}, for $(x,t) \in Q_{h,j}(y)$, we have
\begin{equation}
| \partial_t \Gamma_x (y,t) | \le C d_j^{-N-2},
\end{equation}
which implies
\begin{equation}
\iint_{Q_{h,j}(y) \setminus Q_T} | \partial_t \Gamma_x (y,t) | dxdt \le C h^2d_j^{-1}.
\end{equation}
Furthermore, since the elliptic operator $-\Delta + I$ with the Neumann boundary condition generates a bounded semigroup in $C^0(\overline{\Omega})$, we have
\begin{equation}
\sup_{y \in \Omega} |\partial_t \Gamma_x(y,t)|
\le C \sup_{y \in \Omega} |(-\Delta + I) \bar{\delta}_x(y)|
\le C h^{-N-2}
\label{eq:DtGamma}
\end{equation}
uniformly with respect to $x \in \Omega_h$, which yields
\begin{equation}
\iint_{Q_{h,*}(y) \setminus Q_T} | \partial_t \Gamma_x (y,t) | dxdt
\le C |Q_{h,*}(y) \setminus Q_T| h^{-N-2}
\le C h
\end{equation}
on the innermost set $Q_{h,*}(y)$.
Therefore, we obtain
\begin{equation}
\iint_{Q_{h,T} \setminus Q_T} | \partial_t \Gamma_x (y,t) | dxdt
\le Ch + C \sum_j h^2 d_j^{-1}
\le C h
\end{equation}
owing to \eqref{eq:sum}, where the constant $C$ is independent of $y$ and $h$.
The treatment of the other terms is similar and we can derive
\begin{align}
\max_{x \in \OmghOmg} \iint_{Q_T} | \partial_t \Gamma_x (y,s) | dyds
&\le C |\log h|,\\
\max_{x \in \Omega_h} \iint_{Q_{h,T} \setminus Q_T} | \partial_t \tGamma_x (y,s) | dyds
&\le C h, \\
\max_{y \in \OmghOmg} \iint_{Q_{h,T}} | \partial_t \tGamma_x (y,t) | dxdt
&\le C |\log h|,
\end{align}
with the constant $C$ independent of $h$.
Consequently, we have
\begin{equation}
\| \partial_t \Gamma_x * \tf_h \|_{L^p(Q_{h,T} \setminus Q_T)}
+ \| \Phi \|_{L^p(Q_{h,T})}
\le C \| f_h \|_{L^p(Q_{h,T})}
\label{eq:dmr-3}
\end{equation}
for $p \in (1,\infty)$.
Substituting \eqref{eq:dmr-3} into \eqref{eq:dmr-2}, we can obtain \eqref{eq:dmr-1}.
Hence we complete the proof of \cref{thm:dmr}.
\end{proof}



%
%
%
%

\appendix

\section{Proof of \eqref{eq:sum-trace-T}}
\label{sec:appendix}

We show the auxiliary estimate \eqref{eq:sum-trace-T}.

\begin{proof}[Proof of \eqref{eq:sum-trace-T}.]
Let 
\begin{equation}
\Lambda_{h,T} := \{ j \le J_* \mid D_{h,j} \ne \emptyset \}, \quad
J_1 := \max \Lambda_{h,T}, \quad
\Lambda'_{h,T} := \{ j \le J_* \mid D'_{h,j} \ne \emptyset \}.
\end{equation}
Then, by definition, $\Lambda_{h,T} = \{ 0,1, \dots, J_1 \}$, $\Lambda'_{h,T} = \{ 0,1, \dots, J_1+1 \}$, and $D'_{h,J_1+1} = D_{h,J_1}$ (cf.~Figure~\ref{fig:Dhj}).
Moreover, $D'_{h,j} = D_{h,j-1} \cup D_{h,j} \cup D_{h,j+1}$ and $|D_{h,i} \cap D_{h,j}| = 0$ if $i \ne j$.
Hence, we have
\begin{align}
\sum_j \lambda'_j d_j^{\frac{N}{2}+2} \| F(T) \|_{1, D'_{h,j}}
&
\begin{multlined}[t]
\le 
d_0^{\frac{N}{2}+2} \left( \| F(T) \|_{1, D_{h,0}} + \| F(T) \|_{1, D_{h,1}} \right) \\
\begin{aligned} 
&+ \sum_{j=1}^{J_1-1} d_j^{\frac{N}{2}+2} 
\left( \| F(T) \|_{1, D_{h,j-1}} + \| F(T) \|_{1, D_{h,j}} + \| F(T) \|_{1, D_{h,j+1}} \right) \\
&+ d_{J_1}^{\frac{N}{2}+2} \left( \| F(T) \|_{1, D_{h,J_1-1}} + \| F(T) \|_{1, D_{h,J_1}} \right) 
+ d_{J_1+1}^{\frac{N}{2}+2} \| F(T) \|_{1, D'_{h,J_1+1}}
\end{aligned}
\end{multlined} \\
&\le 3 \cdot 2^{\frac{N}{2}+2} \sum_{j=0}^{J_1} d_j^{\frac{N}{2}+2} \| F(T) \|_{1, D_{h,j}},
\end{align}
which completes the proof.
\end{proof}

\begin{figure}[htb]
\centering
\includegraphics{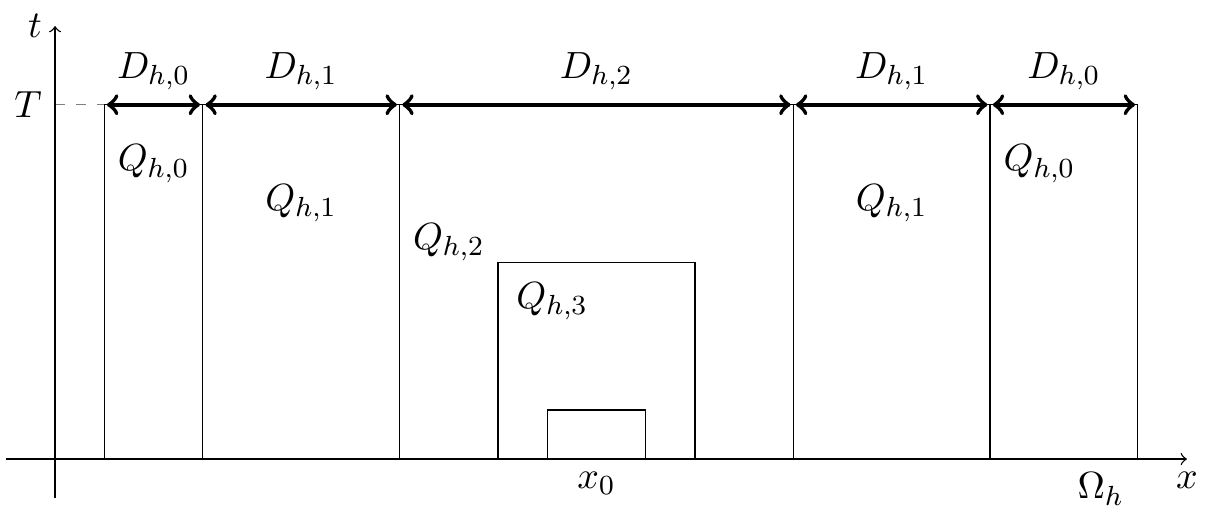}
\caption{Illustration of $D_{h,j}$. In this figure, $J_1 = 2$.}
\label{fig:Dhj}
\end{figure}

\bibliographystyle{plain}
\bibliography{Linfty}

\end{document}